\documentclass[12pt, reqno]{amsart}
\usepackage{}
\usepackage{stmaryrd}
\usepackage{amsmath}
\usepackage{amsfonts}
\usepackage{amssymb}
\usepackage[all,cmtip]{xy}           
\usepackage{xspace}
\usepackage{bbding}
\usepackage{txfonts}
\usepackage[shortlabels]{enumitem}
\usepackage{cite}

\usepackage{tikz}
\usepackage{titletoc}

\usepackage{ifpdf}
\ifpdf
 \usepackage[colorlinks,final,backref=page,hyperindex]{hyperref}
\else
 \usepackage[colorlinks,final,backref=page,hyperindex,hypertex]{hyperref}
\fi
\usepackage[active]{srcltx}

\makeatletter

\begin{document}
\newcommand {\emptycomment}[1]{} 

\baselineskip=15pt
\newcommand{\nc}{\newcommand}
\newcommand{\delete}[1]{}
\nc{\mfootnote}[1]{\footnote{#1}} 
\nc{\todo}[1]{\tred{To do:} #1}

\nc{\mlabel}[1]{\label{#1}}  
\nc{\mcite}[1]{\cite{#1}}  
\nc{\mref}[1]{\ref{#1}}  
\nc{\meqref}[1]{\eqref{#1}} 
\nc{\mbibitem}[1]{\bibitem{#1}} 

\delete{
\nc{\mlabel}[1]{\label{#1}  
{\hfill \hspace{1cm}{\bf{{\ }\hfill(#1)}}}}
\nc{\mcite}[1]{\cite{#1}{{\bf{{\ }(#1)}}}}  
\nc{\mref}[1]{\ref{#1}{{\bf{{\ }(#1)}}}}  
\nc{\meqref}[1]{\eqref{#1}{{\bf{{\ }(#1)}}}} 
\nc{\mbibitem}[1]{\bibitem[\bf #1]{#1}} 
}

\newcommand {\comment}[1]{{\marginpar{*}\scriptsize\textbf{Comments:} #1}}
\nc{\mrm}[1]{{\rm #1}}
\nc{\id}{\mrm{id}}  \nc{\Id}{\mrm{Id}}

\def\a{\alpha}
\def\b{\beta}
\def\bd{\boxdot}
\def\bbf{\bar{f}}
\def\bF{\bar{F}}
\def\bbF{\bar{\bar{F}}}
\def\bbbf{\bar{\bar{f}}}
\def\bg{\bar{g}}
\def\bG{\bar{G}}
\def\bbG{\bar{\bar{G}}}
\def\bbg{\bar{\bar{g}}}
\def\bT{\bar{T}}
\def\bt{\bar{t}}
\def\bbT{\bar{\bar{T}}}
\def\bbt{\bar{\bar{t}}}
\def\bR{\bar{R}}
\def\br{\bar{r}}
\def\bbR{\bar{\bar{R}}}
\def\bbr{\bar{\bar{r}}}
\def\bu{\bar{u}}
\def\bU{\bar{U}}
\def\bbU{\bar{\bar{U}}}
\def\bbu{\bar{\bar{u}}}
\def\bw{\bar{w}}
\def\bW{\bar{W}}
\def\bbW{\bar{\bar{W}}}
\def\bbw{\bar{\bar{w}}}
\def\btl{\blacktriangleright}
\def\btr{\blacktriangleleft}
\def\ci{\circ}
\def\d{\delta}
\def\dd{\diamondsuit}
\def\D{\Delta}
\def\G{\Gamma}
\def\g{\gamma}
\def\k{\kappa}
\def\l{\lambda}
\def\lr{\longrightarrow}
\def\o{\otimes}
\def\om{\omega}
\def\p{\psi}
\def\r{\rho}
\def\ra{\rightarrow}
\def\rh{\rightharpoonup}
\def\lh{\leftharpoonup}
\def\s{\sigma}
\def\st{\star}
\def\ti{\times}
\def\tl{\triangleright}
\def\tr{\triangleleft}
\def\v{\varepsilon}
\def\vp{\varphi}

\newtheorem{thm}{Theorem}[section]
\newtheorem{lem}[thm]{Lemma}
\newtheorem{cor}[thm]{Corollary}
\newtheorem{pro}[thm]{Proposition}
\theoremstyle{definition}
\newtheorem{defi}[thm]{Definition}
\newtheorem{ex}[thm]{Example}
\newtheorem{rmk}[thm]{Remark}
\newtheorem{pdef}[thm]{Proposition-Definition}
\newtheorem{condition}[thm]{Condition}
\newtheorem{question}[thm]{Question}
\renewcommand{\labelenumi}{{\rm(\alph{enumi})}}
\renewcommand{\theenumi}{\alph{enumi}}

\nc{\ts}[1]{\textcolor{purple}{Tianshui:#1}}
\nc{\jie}[1]{\textcolor{blue}{LIJIE:#1}}
\font\cyr=wncyr10

 \title[]{\bf Rota-Baxter operators on Turaev's Hopf group (co)algebras I: Basic definitions and related algebraic structures}

 \author[Ma]{Tianshui Ma\textsuperscript{*}}

 \address{School of Mathematics and Information Science, Henan Normal University, Xinxiang 453007, China}
         \email{matianshui@htu.edu.cn; matianshui@yahoo.com}

 \author[Li]{Jie Li}
 \address{School of Mathematics and Information Science, Henan Normal University, Xinxiang 453007, China}
         \email{{lijie\_0224@163.com}}

 \author[Chen]{Liangyun Chen}
 \address{School of Mathematics and Statistics, Northeast Normal University, Changchun 130024, China}
         \email{chenly640@nenu.edu.cn}

 \author[Wang]{Shuanhong Wang}
 \address{School of Mathematics, Southeast University, Nanjing 210096, China}
         \email{shuanhwang@seu.edu.cn}

  \thanks{\textsuperscript{*}Corresponding author}

\date{\today}

 \begin{abstract}
 We find a natural compatible condition between the Rota-Baxter operator and Turaev's (Hopf) group-(co)algebras, which leads to the concept of Rota-Baxter Turaev's (Hopf) group-(co)algebra. Two characterizations of Rota-Baxter Turaev's group-algebras (abbr. T-algebras) are obtained: one by Atkinson factorization and the other by T-quasi-idempotent elements. The relations among some related Turaev's group algebraic structures (such as (tri)dendriform T-algebras, Zinbiel T-algebras, pre-Lie T-algebras, Lie T-algebras) are discussed, and some concrete examples from the algebras of dimensions 2,3 and 4 are given. At last we prove that Rota-Baxter Poisson T-algebras can produce pre-Poisson T-algebras and Poisson T-algebras can be obtained from pre-Poisson T-algebras.
 \end{abstract}

\keywords{Rota-Baxter (co,bi)algebras, (tri)dendriform  algebras, Zinbiel algebras, (pre-)Lie algebras, Rota-Baxter Poisson algebras, Rota-Baxter Hopf algebras}

\subjclass[2020]{
17B38,
16T05,
16T10,
17B63 
}

 \maketitle

\tableofcontents

\numberwithin{equation}{section}
\allowdisplaybreaks

 \section{Introduction and preliminaries}\label{se:int} Rota-Baxter operators \cite{Ba60} are named after G. Baxter, which is widely used in many fields, such as pure and applied mathematics, and more recently in mathematical physics. Only in the last two years, many important results about Rota-Baxter algebras have appeared, see for example \cite{A03,Go,GTY1,GTY2,MMS,ZGG1,ZGG,ZGM,ZGZ}. Recall that a  Rota-Baxter algebra of weight $\lambda$ \cite{Guo1} is an algebra $\mathcal{A}$ on some base field $K$ together with a linear map $\mathcal{R} : \mathcal{A} \lr \mathcal{A}$ such that
 \begin{eqnarray}
 &\mathcal{R}(a)\mathcal{R}(b)=\mathcal{R}(a\mathcal{R}(b))+\mathcal{R}(\mathcal{R}(a)b)+\lambda \mathcal{R}(ab),& \label{eq:5.35}
 \end{eqnarray}
 for all $a, b\in A$ and $\lambda\in K$. Such a linear operator is called a Rota-Baxter operator of weight $\lambda$ on $A$. One can refer to Guo's book \cite{Guo1} for the detailed theory of Rota-Baxter algebras.

 The notion of (semi-)Hopf $\pi$ (or group)-(co)algebra was introduced by Turaev (see \cite{Tu,Tu1}) from study in homotopy field theory, which also can be used to develop certain invariants of principal $\pi$-bundles over link complements and over 3-manifolds. Examples of (crossed) $\pi$-categories can be constructed from Turaev's Hopf group-coalgebras (abbr. Hopf T-coalgebras): the category of representations of a (crossed) Hopf T-coalgebra is a (crossed) $\pi$-category. Hopf T-(co)algebras have been studied from an algebraic point of view in the literature, see \cite{MLX,MLZ,Tu,Vi02,W09,Z2,Z3} and more recently \cite{BFVV}. When $\pi$ is abelian, a Hopf T-(co)algebra is a $\pi$-colored Hopf algebra due to Ohtsuki \cite{Oh1}. In the case where the underlying group $\pi$ is trivial, a Hopf T-(co)algebra is exactly a classical Hopf algebra.

 As we know, the relationships between Rota-Baxter algebras and Hopf algebras have attracted the attention of many experts, such as, Connes and Kreimer \cite{CK}, Guo, Thibon and Yu \cite{GTY1,GTY2}, in the fields of mathematics and physics. It is worth mentioning that Eq.(\ref{eq:5.39}) (generalizing Rota-Baxter family operators introduced in \cite{EFGBP}) is a very natural compatible condition between Rota-Baxter operators and (semi-)(Hopf) T-algebras (as we know that T-algebras can been seen as an extension of $\pi$-relative associative algebras introduced in \cite{A03}). This leads to one of the main objects (named Rota-Baxter (Hopf) T-algebra) studied here. That is to say, we obtain a class of bialgebralization for Rota-Baxter T-(co)algebras, or equivalently, provide a class of Rota-Baxterization for Hopf T-(co)algebras, i.e., the following diagram is commutative,
 \vspace{0.5cm}
\begin{center}
\unitlength 0.85mm 
\linethickness{0.4pt}
\ifx\plotpoint\undefined\newsavebox{\plotpoint}\fi 
\hspace{-3.5cm}\begin{picture}(83.5,34)(0,0)
\put(-7,6){\small T-(co)algebras}
\put(69,6.5){\small Rota-Baxter T-algebras.}
\put(-8,31.75){\small Hopf T-(co)algebras}
\put(69,30.75){\small Rota-Baxter Hopf T-(co)algebras}
\put(83,12.5){\vector(0,1){13.5}}
\put(12,12.5){\vector(0,1){13.5}}
\put(31.75,34){\small Rota-Baxterization}
\put(31.75,8.75){\small Rota-Baxterization}
\put(83.5,18.5){\small Bialgebraization}
\put(13,18.75){\small Bialgebraization}
\put(29.5,32){\vector(1,0){36.25}}
\put(29.5,6.75){\vector(1,0){36.25}}
\end{picture}
\end{center}\vskip-0.5cm
 This is the motivation of this paper.

 As a special case, Rota-Baxter T-algebras can cover Rota-Baxter family algebras arising naturally in renormalization of quantum field theory introduced by Ebrahimi-Fard, Gracia-Bondia and Patras in \cite[Proposition 9.1]{EFGBP} and proposed by Guo, see also \cite{Guo}. Other research on Rota-Baxter family algebras can be found in \cite{A03,CLZZ,ZG01,ZGM,ZM03}. Rota-Baxter semi-Hopf T-bialgebras can be seen as an  extension of Rota-Baxter bialgebra introduced in \cite{ML}.

 The layout of the paper is as follows. In Section \ref{se:rbta}, we integrate the Rota-Baxter operator into T-(co)algebra and introduce the notion of Rota-Baxter T-(co)algebra (Definition \ref{de:5.14}). Here the family maps are $\{R_{\varphi}: A_{\varphi}\lr A_{\varphi}\}_{\varphi\in \pi}$. Rota-Baxter T-algebra can cover Rota-Baxter family algebra. Two characterizations of Rota-Baxter T-algebras are given. One is a generalization of the Atkinson factorization (Proposition \ref{pro:15.26}). The other is related to the quasi-idempotency (Proposition \ref{pro:15.24}). We provide two approaches to construct Rota-Baxter algebras. We prove that Rota-Baxter T-algebras of weight $\l$ can be obtained by any Rota-Baxter algebras of weight $\l$ (Theorem \ref{thm:15.3}) and more generally, by any Rota-Baxter pair of weight $\l$ (Theorem \ref{thm:15.4}). The relations among (tri)dendriform T-algebras, Rota-Baxter T-algebras and T-algebras are discussed, intuitively, the following diagram is commutative.\vskip-5mm
$$ \xymatrix@C4.5em{
&& \text{tridendriform  }\atop \text{T-algebra}
\ar@2{->}^{\text{Proposition}\ \ref{pro:5.5}}[r] \ar@2{<-}_{\text{Proposition}\  \ref{pro:5.4}}[d]&
\text{ }\atop \text{T-algebras}
\ar@2{<-}_{\text{Proposition}\ \ref{pro:5.5}}[d]&\\
&&\text{Rota-Baxter}\atop \text{T-algebra}\ar@2{->}^{\text{Proposition}\ \ref{pro:5.4}}[r]&\text{dendriform}\atop \text{T-algebra.}
}\qquad\qquad
$$
 As applications, at the same time, we list some concrete examples from the algebras of dimensions 2,3 and 4 (Example \ref{ex:15.12}). Based on Proposition \ref{pro:5.5}, commutative dendriform T-algebra induces commutative T-algebra, but Example \ref{ex:15.19a} shows that noncommutative dendriform T-algebra may induce commutative T-algebra. In Section \ref{se:rbhta}, we present a class of bialgebralization for Rota-Baxter T-(co)algebras, and then introduce the concepts of Rota-Baxter Hopf T-(co)algebras (Definition \ref{de:5.16}) which can be seen as an extension of Rota-Baxter bialgebra (introduced in \cite{ML}) whose examples can be provided by the well-known Radford biproduct playing a central role in the classification of finite dimensional pointed Hopf algebras \cite{AS}. And simultaneously some concrete examples are given (Example \ref{ex:15.11}).

 In Section \ref{se:pi}, we also derive the T-versions of other common algebraic structures related to Rota-Baxter T-algebras and the relations among them, given by the following commutative diagram.
$$ \xymatrix@C4.5em{
&&
&\text{Zinbiel}\atop \text{T-algebra}
\ar@2{<-}_{\text{Proposition}\ \ref{pro:6.7}(+commutative)}[d]& \text{Rota-Baxter Lie}\atop \text{T-algebra}
\ar@2{->}_{\text{Proposition}\ \ref{pro:6.11}}[d]\\
&&
&\text{dendriform}\atop \text{T-algebra}\ar@2{->}^{\text{Proposition}\ \ref{pro:6.2}}[r]&\text{pre-Lie}\atop \text{T-algebra}\\
&& &\ar@2{->}^{\text{Proposition}\  \ref{pro:6.13}}[u] \text{Zinbiel }\atop \text{T-algebra} &\ar@2{<-}^{\text{Proposition}\ \ref{pro:6.15}}[u]\text{Lie}\atop \text{T-algebra} &
}\qquad\qquad\qquad\qquad
$$
 In the meantime we also construct some concrete examples, for example, Example \ref{ex:15.20a} implies that noncommutative dendriform T-algebra may lead to commutative pre-Lie T-algebra, while according to Proposition \ref{pro:6.2}, commutative dendriform T-algebra can lead to commutative (i.e., $a\ast_{p,q} b=b\ast_{q,p} a)$ pre-Lie T-algebra. In Section \ref{se:future}, we illustrate some future work about Rota-Baxter Hopf T-(co)algebras. 
 
 Throughout this paper, $K$ will be a field, and all vector spaces, tensor products, and
 homomorphisms are over $K$.

 \section{Rota-Baxter T-algebras}\label{se:rbta} In this section, we will relate Rota-Baxter operator to T-(co)algebra introduced by Turaev (see \cite{Tu,Tu1}). Moreover Rota-Baxter T-algebra can cover the Rota-Baxter family algebra introduced in \cite[Proposition 9.1]{EFGBP} and proposed by Guo. As applications, we give some concrete examples from the algebras of dimensions 2,3 and 4.

 \subsection{Definition and properties}
 \begin{defi}\label{de:5.14} Let $\pi$ be a semigroup and $\lambda\in K$. A {\bf T-algebra} \cite{Tu} is a pair $(\{A_{\varphi}\}_{\varphi\in \pi}, \{\mu_{p,q}\}_{p, q\in \pi})$, where $\{A_{\varphi}\}_{\varphi\in \pi}$ is a family of vector spaces together with a family of linear maps $\{\mu_{p,q}: A_{p}\otimes A_{q}\lr A_{pq}\}_{p,q\in \pi}$ (write $\mu_{p,q}(a\otimes b)=a\cdot_{p,q}b$) such that
 \begin{eqnarray}
 &\mu_{pq,t}\circ(\mu_{p,q}\otimes \id_{A_{t}})=\mu_{p,qt} \circ(\id_{A_{p}}\otimes \mu_{q,t}),&\label{eq:5.37}
 \end{eqnarray}
 for all $p,q,t\in \pi$.
 A {\bf Rota-Baxter T-algebra of weight $\lambda$} is a quadruple $(\{A_{\varphi}\}_{\varphi\in \pi}, \{\mu_{p,q}\}_{p, q\in \pi}, \{R_{\varphi}\}_{\varphi\in \pi}$, $\lambda)$ (abbr. $(A, R)$), where $(\{A_{\varphi}\}_{\varphi\in \pi}, \{\mu_{p,q}\}_{p, q\in \pi})$ is a T-algebra and $\{R_{\varphi}: A_{\varphi}\lr A_{\varphi}\}_{\varphi\in \pi}$ is a family of linear maps such that
 \begin{eqnarray}
 &\mu_{p,q}\ci(R_p\o R_q)=R_{pq}\ci\mu_{p,q}\ci(R_p\o \id_q+\id_p\o R_q+\l \id_p\o \id_q),&\label{eq:5.39}
 \end{eqnarray}
 for all $p,q\in \pi$.

 If furthermore $\pi$ is a monoid (i.e., a semigroup with a unit 1) and there is a linear map $\eta: K\lr A_{1}$ such that
 \begin{eqnarray}
 &\mu_{p,1}\circ (\id_{A_{p}}\otimes\eta)=\id_{A_{p}}=\mu_{1,p}\circ (\eta\otimes \id_{A_{p}}),&\label{eq:5.38}
 \end{eqnarray}
 then we call Rota-Baxter T-algebra $(A, R)$ {\bf unital}.
 \end{defi}

 \begin{rmk}\label{rmk:5.15}

 (1) If the semigroup $\pi$ contains single element $e$, then the Rota-Baxter T-algebra is exactly the Rota-Baxter algebra of weight $\l$ in \cite{Ro1}.

 (2) If $(\mathcal{A},\mu,\eta)$ is an associative algebra, then Rota-Baxter T-algebra $(\{A_{\varphi}=\mathcal{A}\}, \{\mu_{p, q}=\mu\}, \{R_{\varphi}\}, \lambda)$ is the Rota-Baxter family algebra of weight $\l$ introduced in \cite[Proposition 9.1]{EFGBP}, proposed by Guo and studied in \cite{CLZZ,Guo,ZG01}. 

 (3) If we set $\{A_{\varphi}=\mathcal{A}\}_{\varphi\in \pi}$, then Eq.(\ref{eq:5.37}) is exactly the \cite[Eq.(2.1)]{A03}, i.e., the condition for $\pi$-relative.
 \end{rmk}

 Rota-Baxter T-algebras can be constructed by Rota-Baxter T-algebras in the following way.

 \begin{lem}\label{lem:15.25} Let $(A,R)$ be a Rota-Baxter T-algebra of weight $\lambda$. Define
 \begin{eqnarray}
  &\tilde{R}:=\{\tilde{R}_{\varphi}=-\lambda id_{A_{\varphi}} -R_{\varphi}\}.&\label{eq:16.00}
 \end{eqnarray}
  Then $(A, \tilde{R})$ is also a Rota-Baxter T-algebra of weight $\lambda$.
 \end{lem}

 \begin{proof} For any $a\in A_{p}$, $b\in A_{q}$ and $p,q\in \pi$
 \begin{eqnarray*}
 \tilde{R}_{p}(a)\cdot_{p,q}\tilde{R}_{q}(b)&=&(-\lambda a-R_{p}(a))\cdot_{p,q}(-\lambda b-R_{q}(b))\\
 &=&\lambda^{2} a\cdot_{p,q}b+\lambda R_{p}(a)\cdot_{p,q}b+\lambda a\cdot_{p,q}R_{q}(b)+R_{p}(a)\cdot_{p,q}R_{q}(b)\\
 &=&(\lambda id_{A_{pq}} +R_{pq})(R_{p}(a)\cdot_{p,q}b+a\cdot_{p,q}R_{q}(b)+\lambda a\cdot_{p,q}b)\\
 &=&\tilde{R}_{pq}(\tilde{R}_{p}(a)\cdot_{p,q}b+a\cdot_{p,q}\tilde{R}_{q}(b)+\lambda a\cdot_{p,q}b).
 \end{eqnarray*}
 This is what we need.
 \end{proof}

 In what follows, we provide two characterizations of Rota-Baxter T-algebras. First we consider the Atkinson factorization \cite{At} for Rota-Baxter T-algebras.

 \begin{pro}\label{pro:15.26} Let $\pi$ be a semigroup, $A$ be a T-algebra and $R=\{R_{\varphi}: A_{\varphi}\lr A_{\varphi}\}$ be a family of linear maps. Assume that $\lambda\in K$ have no zero divisor in $\{A_{\varphi}\}_{\varphi\in \pi}$ (which means that if for $\l\neq 0$ and $c\in A_{\vp}$ such that $\l c=0$, then $c=0$) and define $\tilde{R}$ as in Eq.(\ref{eq:16.00}). Then $(A,R)$ is a Rota-Baxter T-algebra of weight $\lambda\neq 0$ if and only if, for any $a\in A_{p}$, $b\in A_{q}$ and $p,q\in \pi$, there is an element $c\in A_{pq}$ such that
 \begin{eqnarray}
  &R_{p}(a)\cdot_{p,q}R_{q}(b)=R_{pq}(c),\quad \tilde{R}_{p}(a)\cdot_{p,q}\tilde{R}_{q}(b)=-\tilde{R}_{pq}(c).&\label{eq:16.01}
 \end{eqnarray}
 \end{pro}

 \begin{proof} ($\Longrightarrow$) Let $(A,R)$ be a Rota-Baxter T-algebra of weight $\lambda\neq 0$. Then, for any $a\in A_{p}$, $b\in A_{q}$ and $p,q\in \pi$, we have
 \begin{eqnarray*}
 R_{p}(a)\cdot_{p,q}R_{q}(b)=R_{pq} \big(R_{p}(a)\cdot_{p,q}b+a\cdot_{p,q}R_{q}(b)+\lambda a\cdot_{p,q}b \big).
 \end{eqnarray*}
 Let $c= R_{p}(a)\cdot_{p,q}b+a\cdot_{p,q}R_{q}(b)+\lambda a\cdot_{p,q}b $. Then the above equality gives
 \begin{eqnarray*}
 R_{p}(a)\cdot_{p,q}R_{q}(b)=R_{pq}(c).
 \end{eqnarray*}
 Based on the proof of Lemma \ref{lem:15.25}, we have
 \begin{eqnarray*}
 \tilde{R}_{p}(a)\cdot_{p,q}\tilde{R}_{q}(b)=-\tilde{R}_{pq}(c).
 \end{eqnarray*}

 ($\Longleftarrow$) For any $a\in A_{p}$, $b\in A_{q}$ and $p,q\in \pi$. Suppose that there exists an element $c\in A_{pq}$ such that $R_{p}(a)\cdot_{p,q}R_{q}(b)=R_{pq}(c),\  \tilde{R}_{p}(a)\cdot_{p,q}\tilde{R}_{q}(b)=-\tilde{R}_{pq}(c)$. Then we have
 \begin{eqnarray*}
 -\lambda c &=& R_{pq}(c)+\tilde{R}_{pq}(c)\\
 &=&R_{p}(a)\cdot_{p,q}R_{q}(b)-\tilde{R}_{p}(a)\cdot_{p,q}\tilde{R}_{q}(b)\\
 &=&R_{p}(a)\cdot_{p,q}R_{q}(b)-\lambda(R_{p}(a)\cdot_{p,q}b+a\cdot_{p,q}R_{q}(b)+\lambda a\cdot_{p,q}b)-R_{p}(a)\cdot_{p,q}R_{q}(b)\\
 &=&-\lambda(R_{p}(a)\cdot_{p,q}b+a\cdot_{p,q}R_{q}(b)+\lambda a\cdot_{p,q}b).
 \end{eqnarray*}
 So we get
 \begin{eqnarray*}
 c &=& R_{p}(a)\cdot_{p,q}b+a\cdot_{p,q}R_{q}(b)+\lambda a\cdot_{p,q}b.
 \end{eqnarray*}
 Then
 \begin{eqnarray*}
 R_{p}(a)\cdot_{p,q}R_{q}(b)=R_{pq}(R_{p}(a)\cdot_{p,q}b+a\cdot_{p,q}R_{q}(b)+\lambda a\cdot_{p,q}b),
 \end{eqnarray*}
 as desired.
 \end{proof}

 Now we give the two useful definitions.
 \begin{defi}\label{de:15.23} A family of linear maps $R=\{R_{\varphi}: A_{\varphi}\longrightarrow A_{\varphi}\}_{\varphi\in \pi}$ is called {\bf T-idempotent} if $R_{\varphi}^{2}=R_{\varphi},~ \forall~\varphi\in \pi$. A family of linear maps $R=\{R_{\varphi}: A_{\varphi}\longrightarrow A_{\varphi}\}_{\varphi\in \pi}$ is called {\bf T-quasi-idempotent of weight $\lambda$} if $R_{\varphi}^{2}=-\lambda R_{\varphi}, \forall~\varphi\in \pi$.
 \end{defi}

 Then we can obtain an equivalent description of Rota-Baxter T-algebras by T-quasi-idempotency.
 \begin{pro}\label{pro:15.24} Let $\pi$ be a monoid, $A$ be a unital T-algebra and $R=\{R_{\varphi}: A_{\varphi}\longrightarrow A_{\varphi}\}_{\varphi\in\pi}$ be left $A$-linear in the sense that $R_{pq}(a\cdot_{p,q} b)=a\cdot_{p,q} R_{q}(b), \forall~a\in A_{p}, b\in A_{q}, p,q\in \pi$. Then $(A,R)$ is a Rota-Baxter T-algebra of weight $\lambda$ if and only if $R$ is T-quasi-idempotent of weight $\lambda$.
 \end{pro}

 \begin{proof} Under the condition of left $A$-linearity of $R$, for $a\in A_{p}$, $b\in A_{q}$ and $p,q\in \pi$, we can get
 \begin{eqnarray*}
 &&R_{p}(a)\cdot_{p,q}R_{q}(b)=R_{pq}(R_{p}(a)\cdot_{p,q}b+a\cdot_{p,q}R_{q}(b)+\lambda a\cdot_{p,q}b)\\
 &&\Leftrightarrow a\cdot_{p,q}R_{q}^{2}(b)+\lambda a\cdot_{p,q}R_{q}(b)=0 \Leftrightarrow R_{q}^{2} +\lambda  R_{q} =0,
 \end{eqnarray*}
 finishing the proof.
 \end{proof}

 \begin{pro}\label{pro:15.27} Let $(A,R)$ be a Rota-Baxter T-algebra of weight $\lambda$. If $R$ is T-idempotent, then
 \begin{eqnarray}
  &(1+\lambda)R_{pq}(a\cdot_{p,q}R_{q}(b))=0, \quad (1+\lambda)R_{pq}(R_{p}(a)\cdot_{p,q}b)=0,&\label{eq:16.02}\\
  &(1+\lambda)(R_{p}(a)\cdot_{p,q}R_{q}(b)-\lambda R_{pq}(a\cdot_{p,q}b))=0,&\label{eq:16.04}
 \end{eqnarray}
 for all $a\in A_{p}$, $b\in A_{q}$ and $p,q\in \pi$.
 \end{pro}

 \begin{proof} For any $a\in A_{p}$, $b\in A_{q}$ and $p,q\in \pi$,
 \begin{eqnarray*}
 R_{p}(a)\cdot_{p,q}R_{q}(b) &=& R_{p}(a)\cdot_{p,q}R_{q}^{2}(b)=R_{p}(a)\cdot_{p,q}R_{q}(R_{q}(b))\\
 &=&R_{pq}(R_{p}(a)\cdot_{p,q}R_{q}(b)+a\cdot_{p,q}R_{q}^{2}(b)+\lambda a\cdot_{p,q}R_{q}(b))\\
 &=&R_{pq}(R_{p}(a)\cdot_{p,q}R_{q}(b))+R_{pq}(a\cdot_{p,q}R_{q}(b))+\lambda R_{pq} (a\cdot_{p,q}R_{q}(b)).
 \end{eqnarray*}
 Applying $R_{pq}$ to the above equality and using $R_{pq}^{2}=R_{pq}$, we obtain
 \begin{eqnarray*}
 (1+\lambda)R_{pq}(a\cdot_{p,q}R_{q}(b))=0.
 \end{eqnarray*}
 Similarly,
\begin{eqnarray*}
 R_{p}(a)\cdot_{p,q}R_{q}(b)
 =R_{pq}(R_{p}(a)\cdot_{p,q}b)+R_{pq}(R_{p}(a)\cdot_{p,q}R_{q}(b))+\lambda R_{pq} (R_{p}(a)\cdot_{p,q}b).
 \end{eqnarray*}
 Applying $R_{pq}$ again and using $R_{pq}^{2}=R_{pq}$, one has
 \begin{eqnarray*}
 (1+\lambda)R_{pq}(R_{p}(a)\cdot_{p,q}b)=0.
 \end{eqnarray*}
 According to the above two equations, Eq.(\ref{eq:16.04}) is obvious.
 \end{proof}

 \subsection{Constructions of Rota-Baxter T-algebras}\label{se:crbta} In this subsection,  we give some concrete examples from the algebras of dimensions 2,3 and 4.

 Let $\pi$ be a nonempty set and $\mathcal{A}$ an algebra. We denote $\mathcal{A}[\pi]=\mathcal{A}\otimes K\pi, \mathcal{A}\varphi=\mathcal{A}\otimes K\varphi, \forall\varphi\in\pi$. If, further, $\pi$ is a semigroup and $\mathcal{A}$ is an algebra, then $\mathcal{A}[\pi]$ is a T-algebra with the tensor product algebra structure, i.e.,
 \begin{eqnarray*}
 &\mu_{p,q}:\mathcal{A} p\otimes \mathcal{A} q\longrightarrow \mathcal{A}(pq)&\\
 &\qquad\quad hp\otimes gq\longmapsto (hg)(pq)&
 \end{eqnarray*}
 for all $h,g\in \mathcal{A}$ and $p,q\in \pi$ (\cite{Tu,Tu1}).

 The following result is an analogy of \cite[Proposition 9.2]{EFGBP}.

 \begin{thm}\label{thm:15.3}  Let $\pi$ be a semigroup, $\mathcal{A}$ an algebra and $\mathcal{R}: \mathcal{A}\longrightarrow \mathcal{A}$ a linear map. For any $\varphi \in\pi$, define
 \begin{eqnarray*}
 &&R_{\varphi}:\mathcal{A}\varphi \longrightarrow \mathcal{A}\varphi\\
 &&\qquad h\varphi \longmapsto \mathcal{R}(h)\varphi
 \end{eqnarray*}
 for all $h\in \mathcal{A}$. Then $(\mathcal{A},\mathcal{R})$ is a Rota-Baxter algebra of weight $\lambda$ if and only if $(\mathcal{A}[\pi],R=\{R_{\varphi}\})$ is a Rota-Baxter T-algebra of weight $\lambda$.
 \end{thm}

 \begin{proof} For all $h,g\in \mathcal{A}$ and $p,q\in \pi$,
 \begin{eqnarray*}
 &&\hspace{-13mm}R_{pq}(R_{p}(hp)\cdot_{p,q}gq+hp\cdot_{p,q}R_{q}(gq)+\lambda(hp)\cdot_{p,q}(gq))\\
 &=&R_{pq}(\mathcal{R}(h)p \cdot_{p,q}gq+hp\cdot_{p,q}\mathcal{R}(g)q
 +\lambda(hp)\cdot_{p,q}(gq))\\
 &=&R_{pq}((\mathcal{R}(h)g)(pq)+(h\mathcal{R}(g))(pq)+\lambda(hg)(pq))\\
 &=&\mathcal{R}(\mathcal{R}(h)g+h\mathcal{R}(g)+\lambda hg)(pq)
 \end{eqnarray*}
 and
 \begin{eqnarray*}
 R_{p}(hp)\cdot_{p,q}R_{q}(gq)
 &=&(\mathcal{R}(h)p)\cdot_{p,q}(\mathcal{R}(g)q)=(\mathcal{R}(h)\mathcal{R}(g))(pq).
 \end{eqnarray*}
 Therefore we finish the proof.
 \end{proof}

 \begin{rmk}\label{rmk:15.3-1} Theorem \ref{thm:15.3} shows that a Rota-Baxter T-algebra $(\mathcal{A}[\pi],R)$ if $(\mathcal{A},\mathcal{R})$ is a Rota-Baxter algebra of weight $\lambda$.
 \end{rmk}

 Next we provide another approach to construct Rota-Baxter T-algebras.

 \begin{defi}\label{de:15.2} Let $(\mathcal{A},\mathcal{R})$ and $(\mathcal{A},\mathcal{R}')$ be two Rota-Baxter algebras of weight $\lambda$. A pair $(\mathcal{R},\mathcal{R}')$ on $\mathcal{A}$ is called a {\bf Rota-Baxter pair of weight $\lambda$} if it satisfies
 \begin{eqnarray}
 \mathcal{R}(h)\mathcal{R}'(g)=\mathcal{R}'(\mathcal{R}(h)g+h \mathcal{R}'(g)+\lambda hg),\label{eq:15.1}\\
 \mathcal{R}'(h)\mathcal{R}(g)=\mathcal{R}'(\mathcal{R}'(h)g+h \mathcal{R}(g)+\lambda hg),\label{eq:15.2}
 \end{eqnarray}
 for all $h,g\in \mathcal{A}$.
 \end{defi}

 \begin{rmk}\label{rmk:15.2-1} (1) In particular, if $(\mathcal{A},\mathcal{R})$ is a Rota-Baxter algebra of weight $\lambda$, then $(\mathcal{R},\mathcal{R})$ is also a Rota-Baxter pair of weight $\lambda$.

 (2) If furthermore $\mathcal{A}$ is commutative, then $(\mathcal{R},\mathcal{R}')$ is a Rota-Baxter pair if Eq.(\ref{eq:15.1}) or Eq.(\ref{eq:15.2}) holds.
 \end{rmk}

 Since this is the first paper for Rota-Baxter T-algebras, we present examples as detailed as possible for the applications in subsequent research.

 The following Rota-Baxter operators on the algebra of dimension 2 were studied by de Bragan$\c{c}$a in \cite{de}, by An and Bai in \cite{AB}, by Gubarev in \cite{G} and by Makhlouf, Silvestrov and the first author in \cite{MMS}.
 \begin{ex}\label{ex:15.15} We consider the 2-dimensional algebra $\mathcal{A}=K\{u_1,u_2|u_1\cdot u_i=u_i\cdot u_1=u_i,i=1,2, u_2\cdot u_2 =u_2\}$. Then the Rota-Baxter operators on $\mathcal{A}$ of weight $\lambda$ are
 \begin{enumerate}
 \item $\mathcal{R}(u_1)=-\lambda u_1, \  \mathcal{R}(u_2)=0, $
 \item $\mathcal{R}(u_1)=0, \  \mathcal{R}(u_2)=-\lambda u_2, $
 \item $\mathcal{R}(u_1)=-\lambda u_1, \  \mathcal{R}(u_2)=-\lambda u_2, $
 \item $\mathcal{R}(u_1)=-\lambda u_2, \  \mathcal{R}(u_2)=-\lambda u_2, $
 \item $\mathcal{R}(u_1)=-\lambda u_1, \  \mathcal{R}(u_2)=-\lambda u_1, $
 \item $\mathcal{R}(u_1)=\lambda u_2, \  \mathcal{R}(u_2)=0, $
 \item $\mathcal{R}(u_1)=0, \  \mathcal{R}(u_2)=\lambda u_1-\lambda u_2, $
 \item $\mathcal{R}(u_1)=-2 \lambda u_1+\lambda u_2, \  \mathcal{R}(u_2)=-\lambda u_1, $
 \item $\mathcal{R}(u_1)=- \lambda u_1-\lambda u_2, \  \mathcal{R}(u_2)=-\lambda u_2, $
 \item $\mathcal{R}(u_1)=- \lambda u_1+\lambda u_2, \  \mathcal{R}(u_2)=0, $
 \item $\mathcal{R}(u_1)= \lambda u_1-\lambda u_2, \  \mathcal{R}(u_2)= \lambda u_1-\lambda u_2. $
 \end{enumerate}
 \end{ex}

 Now we list some Rota-Baxter operators on the algebras of dimensions 3 and 4 given by Makhlouf, Silvestrov and the first author in \cite{MMS}.

 \begin{ex}\label{ex:15.16} For the 3-dimensional algebra $\mathcal{A}=K\{u_1,u_2,u_3|u_1\cdot u_i=u_i\cdot u_1=u_i,i=1,2,3, u_2\cdot u_2 =u_2, u_2\cdot u_3 =u_3\cdot u_2 =u_3, u_3\cdot u_3 =0\}$, the Rota-Baxter operators of weight $\lambda$ on $\mathcal{A}$ are
 \begin{enumerate}
 \item $\mathcal{R}(u_1)=-\lambda u_1, \  \mathcal{R}(u_2)=0,\mathcal{R}(u_3)=0,$
 \item $\mathcal{R}(u_1)=0, \  \mathcal{R}(u_2)=-\lambda u_2,\mathcal{R}(u_3)=0,$
 \item $\mathcal{R}(u_1)=0, \  \mathcal{R}(u_2)=0,\mathcal{R}(u_3)=-\lambda u_3,$
 \item $\mathcal{R}(u_1)=-\lambda u_1, \  \mathcal{R}(u_2)=-\lambda u_2,\mathcal{R}(u_3)=0,$
 \item $\mathcal{R}(u_1)=0, \  \mathcal{R}(u_2)=-\lambda u_2,\mathcal{R}(u_3)=-\lambda u_3,$
 \item $\mathcal{R}(u_1)=-\lambda u_1, \  \mathcal{R}(u_2)=0,\mathcal{R}(u_3)=-\lambda u_3,$
 \item $\mathcal{R}(u_1)=-\lambda u_1, \  \mathcal{R}(u_2)=-\lambda u_2,\mathcal{R}(u_3)=-\lambda u_3,$
 \item $\mathcal{R}(u_1)=-\lambda u_2, \  \mathcal{R}(u_2)=-\lambda u_2,\mathcal{R}(u_3)=0,$
 \item $\mathcal{R}(u_1)=-\lambda u_2, \  \mathcal{R}(u_2)=-\lambda u_2,\mathcal{R}(u_3)=-\lambda u_3,$
 \item $\mathcal{R}(u_1)=\lambda u_2, \  \mathcal{R}(u_2)=0,\mathcal{R}(u_3)=0,$
 \item $\mathcal{R}(u_1)=\lambda u_2, \  \mathcal{R}(u_2)=0,\mathcal{R}(u_3)=-\lambda u_3,$
 \item $\mathcal{R}(u_1)=0, \  \mathcal{R}(u_2)=\lambda u_1-\lambda u_2,\mathcal{R}(u_3)=0,$
 \item $\mathcal{R}(u_1)=0, \  \mathcal{R}(u_2)=\lambda u_1-\lambda u_2,\mathcal{R}(u_3)=-\lambda u_3,$
 \item $\mathcal{R}(u_1)=-2\lambda u_1+\lambda u_2, \  \mathcal{R}(u_2)=-\lambda u_1,\mathcal{R}(u_3)=0,$
 \item $\mathcal{R}(u_1)=-2\lambda u_1+\lambda u_2, \  \mathcal{R}(u_2)=-\lambda u_1,\mathcal{R}(u_3)=-\lambda u_3,$
 \item $\mathcal{R}(u_1)=-\lambda u_1-\lambda u_2, \  \mathcal{R}(u_2)=-\lambda u_2,\mathcal{R}(u_3)=0,$
 \item $\mathcal{R}(u_1)=-\lambda u_1-\lambda u_2, \  \mathcal{R}(u_2)=-\lambda u_2,\mathcal{R}(u_3)=-\lambda u_3,$
 \item $\mathcal{R}(u_1)=-\lambda u_1+\lambda u_2, \  \mathcal{R}(u_2)=0,\mathcal{R}(u_3)=0,$
 \item $\mathcal{R}(u_1)=-\lambda u_1+\lambda u_2, \  \mathcal{R}(u_2)=0,\mathcal{R}(u_3)=-\lambda u_3,$
 \item $\mathcal{R}(u_1)=-\lambda u_1, \  \mathcal{R}(u_2)=-\lambda u_1,\mathcal{R}(u_3)=0,$
 \item $\mathcal{R}(u_1)=-\lambda u_1, \  \mathcal{R}(u_2)=-\lambda u_1,\mathcal{R}(u_3)=-\lambda u_3,$
 \item $\mathcal{R}(u_1)=\lambda u_1-\lambda u_2, \  \mathcal{R}(u_2)=\lambda u_1-\lambda u_2,\mathcal{R}(u_3)=0,$
 \item $\mathcal{R}(u_1)=\lambda u_1-\lambda u_2, \  \mathcal{R}(u_2)=\lambda u_1-\lambda u_2,\mathcal{R}(u_3)=-\lambda u_3.$
 \end{enumerate}
 \end{ex}

 \begin{ex}\label{ex:15.17} For the Taft-Sweedler algebra $\mathcal{A}=T_2=K\{u_1=1, u_2=g, u_3=x, u_4=gx|g^2=1, x^2=0, xg=-gx\}$, its multiplication can be given by the following table.
 \[
 \begin{array}{|c|c|c|c|c|}
  \hline
  \cdot & u_1& u_2 & u_3 & u_4  \\ \hline
   u_1& u_1& u_2 & u_3 & u_4 \\ \hline
   u_2 &u_2 & u_1 & u_4 & u_3 \\ \hline
   u_3 &u_3 & -u_4 & 0 & 0  \\ \hline
   u_4 &u_4 & -u_3 & 0 & 0  \\ \hline
 \end{array}~~.
 \]

 The Rota-Baxter operators on $\mathcal{A}$ of weight $\lambda$ are given by
 \begin{enumerate}
 \item $\mathcal{R}(u_1)=0, \  \mathcal{R}(u_2)=0, \mathcal{R}(u_3)=-\lambda u_3, \  \mathcal{R}(u_4)=-\lambda u_4, \  $
 \item $\mathcal{R}(u_1)=-\lambda u_1, \  \mathcal{R}(u_2)=-\lambda u_2, \mathcal{R}(u_3)=0, \  \mathcal{R}(u_4)=0, \  $
 \item $\mathcal{R}(u_1)=-\lambda u_1, \  \mathcal{R}(u_2)=-\lambda u_2, \mathcal{R}(u_3)=-\lambda u_3, \  \mathcal{R}(u_4)=-\lambda u_4.\ $
 \item $\mathcal{R}(u_1)=0,
     \\\mathcal{R}(u_2)=-p_1u_1+p_1u_2-\frac{(\lambda+p_1)(\lambda+p_1+p_2)}{p_3}u_3+\frac{(\lambda+p_1)
     (\lambda+p_2)}{p_3}u_4,
     \\ \mathcal{R}(u_3)=-p_3u_1+p_3u_2-(2\lambda+p_1+p_2)u_3+(\lambda+p_2)u_4,
     \\ \mathcal{R}(u_4)=-p_3u_1+p_3u_2-(\lambda+p_1+p_2)u_3+p_2u_4,$
 \item $\mathcal{R}(u_1)=-\lambda u_1, \\\mathcal{R}(u_2)=(\lambda+p_1)u_1+p_1u_2-\frac{(\lambda+p_1)(\lambda+p_1+p_2)}{p_3}u_3+\frac{(\lambda+p_1)
     (\lambda+p_2)}{p_3}u_4,
     \\ \mathcal{R}(u_3)=p_3u_1+p_3u_2-(2\lambda+p_1+p_2)u_3+(\lambda+p_2)u_4,
     \\ \mathcal{R}(u_4)=p_3u_1+p_3u_2-(\lambda+p_1+p_2)u_3+p_2u_4,$
 \item $\mathcal{R}(u_1)=-\lambda u_1,~~\mathcal{R}(u_2)=\lambda u_1+p_1u_3+\frac{p_1p_2}{\lambda+p_2}u_4,
     \\ \mathcal{R}(u_3)=-(\lambda+p_2)u_3-p_2u_4,~~
     \mathcal{R}(u_4)=(\lambda+p_2)u_3+p_2u_4,$
 \item $\mathcal{R}(u_1)=-\lambda u_1,~~\mathcal{R}(u_2)=\lambda u_1+\frac{\lambda(\lambda+p_1)}{p_2}u_3+\frac{\lambda(\lambda+p_1)}{p_2}u_4,
    \\ \mathcal{R}(u_3)=-p_2u_1-p_2u_2-(2\lambda+p_1)u_3-(\lambda+p_1)u_4,
    \\ \mathcal{R}(u_4)=p_2u_1+p_2u_2+(\lambda+p_1)u_3+p_1u_4,$
 \item $\mathcal{R}(u_1)=\frac{1}{2}\lambda u_1-\frac{1}{2}\lambda u_2 +p_1u_3+p_2u_4,~~\mathcal{R}(u_2)=\frac{1}{2}\lambda u_1-\frac{1}{2}\lambda u_2-p_2u_3+p_1u_4,\
    \\ \mathcal{R}(u_3)=-\frac{1}{2}\lambda u_3-\frac{1}{2}\lambda u_4, ~~\mathcal{R}(u_4)=-\frac{1}{2}\lambda u_3-\frac{1}{2}\lambda u_4.$
 \end{enumerate}
 \end{ex}

 On the basis of the above examples, we have

 \begin{ex}\label{ex:15.5} Let $\mathcal{A}$ be the 2-dimensional algebra defined in Example \ref{ex:15.15}. Then
 \begin{eqnarray*}
 && (a,c),\quad (a,d),\quad (a,e),\quad (a,j),\quad (b,c),\quad (b,d),\quad (b,j),\quad (c,d),\quad (c,g),\quad (c,j),\qquad\\
 && (d,i),\quad (d,j),\quad (d,k),\quad (e,j),\quad (f,h),\quad (f,j),\quad (g,j),\quad (h,j),\quad (i,j),\quad (i,k),\qquad\\
 && (c,b),\quad (d,c),\quad (e,a),\quad (e,b),\quad (e,c),\quad (e,d),\quad (f,c),\quad (f,d),\quad (g,b),\quad (g,c),\qquad\\
 && (g,d),\quad(h,c),\quad (h,d),\quad (h,f),\quad (i,c),\quad (i,d),\quad (j,c),\quad (j,d),\quad (j,f),\quad (j,h),\qquad\\
 && (k,c),\quad(k,d),\quad (k,i)\qquad
 \end{eqnarray*}
 are the Rota-Baxter pairs $(\mathcal{R}, \mathcal{R}')$ on $\mathcal{A}$ of weight $\lambda$, where $(a,c)$ represents $(\mathcal{R}, \mathcal{R}')$, and $\mathcal{R}, \mathcal{R}'$ take the case $(a)$ and case $(c)$ respectively in Example \ref{ex:15.15}, i.e.,
 \begin{eqnarray*}
 &\mathcal{R}(u_1)=-\lambda u_1,\ \mathcal{R}(u_2)=0,\ \qquad \mathcal{R}'(u_1)=-\lambda u_1,\  \mathcal{R}'(u_2)=-\lambda u_2.&
 \end{eqnarray*}
 Others are similar.
 \end{ex}

 \begin{ex}\label{ex:15.6} Let $\mathcal{A}$ be the 3-dimensional algebra defined in Example \ref{ex:15.16}. Then
 \begin{eqnarray*}
 &&(a,d),\quad (a,g),\quad (a,h),\quad (a,i),\quad (a,r),\quad (a,t),\quad (a,u),\quad (b,d),\quad (b,g),\quad (b,h),\quad (b,i),\qquad\\
 && (b,l),\quad (b,r),\quad (c,f),\quad (c,g),\quad (c,i),\quad (c,r),\quad (c,s),\quad (c,u),\quad (d,g),\quad (d,h),\quad (d,i),\qquad\\
 && (d,l),\quad (d,r),\quad (e,g),\quad (e,i),\quad (e,l),\quad (e,m),\quad (e,r),\quad (e,s),\quad (e,u),\quad (f,g),\quad (f,i),\qquad\\
 && (f,r),\quad (f,s),\quad (f,u),\quad (g,i),\quad (g,l),\quad (g,m),\quad (g,r),\quad (g,s),\quad (g,v),\quad (g,w),\quad (h,i),\qquad\\
 && (h,r),\quad (h,v),\quad (i,q),\quad (i,r),\quad (i,s),\quad (i,v),\quad (i,w),\quad (j,n),\quad (j,o),\quad (j,r),\quad (k,o),\qquad\\
 && (k,r),\quad (k,s),\quad (l,r),\quad (m,r),\quad (m,s),\quad (n,o),\quad (n,r),\quad (o,r),\quad
 (o,s),\quad (p,r),\quad (p,v),\qquad\\
 && (q,r),\quad (q,s),\quad (q,v),\quad (q,w),\quad (t,u),\quad (d,b),\quad (e,c),\quad (f,c),\quad
 (g,c),\quad (g,e),\quad (h,d),\qquad\\
 && (h,g),\quad (i,c),\quad (i,g),\quad (j,g),\quad (j,h),\quad (j,i),\quad (k,c),\quad (k,g),\quad
 (k,i),\quad (l,b),\quad (l,d),\qquad\\
 && (l,e),\quad (l,g),\quad (l,h),\quad (l,i),\quad (m,c),\quad (m,e),\quad (m,g),\quad (m,i),\quad
 (m,l),\quad (n,d),\quad (n,g),\qquad\\
 && (n,h),\quad (n,i),\quad (n,j),\quad (o,c),\quad (o,g),\quad (o,i),\quad (o,k),\quad (p,d),\quad
 (p,g),\quad (p,h),\quad (p,i),\qquad\\
 && (q,c),\quad (q,g),\quad (q,i),\quad (r,d),\quad (r,g),\quad (r,h),\quad (r,i),\quad (r,j),\quad
 (r,n),\quad (r,o),\quad (s,c),\qquad\\
 && (s,g),\quad (s,i),\quad (s,k),\quad (s,o),\quad (s,r),\quad (t,a),\quad (t,d),\quad (t,g),\quad
 (t,h),\quad (t,i),\quad (t,r),\qquad\\
 && (u,c),\quad (u,f),\quad (u,g),\quad (u,i),\quad (u,r),\quad (u,s),\quad (v,d),\quad (v,g),\quad
 (v,h),\quad (v,i),\quad (v,p),\qquad\\
 && (v,r),\quad (w,c),\quad (w,g),\quad (w,i),\quad (w,q),\quad (w,r),\quad (w,s),\quad (w,v)\qquad
 \end{eqnarray*}
 are the Rota-Baxter pairs $(\mathcal{R}, \mathcal{R}')$ on $\mathcal{A}$ of weight $\lambda$, where $a-w$ represent the cases $(a)-(w)$ in Example \ref{ex:15.16}.
 \end{ex}

 \begin{ex}\label{ex:15.7} Let $\mathcal{A}$ be the 4-dimensional algebra defined in Example \ref{ex:15.17}. Then
 \begin{eqnarray*}
 &(a,c),\quad (b,c),\quad (c,a),\quad (d,c),\quad (e,c),\quad (f,c),\quad (g,c),\quad (h,c)\quad &
 \end{eqnarray*}
 are the Rota-Baxter pairs $(\mathcal{R}, \mathcal{R}')$ on $\mathcal{A}$ of weight $\lambda$, where $a-h$ represent the cases $(a)-(h)$ in Example \ref{ex:15.17}.
 \end{ex}

 \begin{rmk} In Examples \ref{ex:15.5}-\ref{ex:15.7}, we leave out the case of $(R, R)$ since it can be seen as the special case of Theorem \ref{thm:15.3}.
 \end{rmk}

 \begin{thm}\label{thm:15.4} Let $\pi=\{1,q\}$ be a monoid with a unit $1$ and $q^{2}=q$. Assume that $(\mathcal{A},\mathcal{R})$ and $(\mathcal{A},\mathcal{R}')$ are Rota-Baxter algebras of weight $\lambda$. Define
 \begin{eqnarray*}
 &&R_{1_{\pi}}:\mathcal{A} 1_{\pi}  \longrightarrow \mathcal{A} 1_{\pi} \qquad\qquad R_{q}:\mathcal{A} q \longrightarrow \mathcal{A} q \qquad\\
 &&\qquad\  h 1_{\pi}  \longmapsto \mathcal{R}(h) 1_{\pi} \quad\qquad\qquad   h q  \longmapsto \mathcal{R}'(h) q,\qquad
 \end{eqnarray*}
 then $(\mathcal{A}[\pi],\{R_{\varphi}\},\lambda)$ is a Rota-Baxter T-algebra of weight $\lambda$ if and only if $(\mathcal{R},\mathcal{R}')$ is a Rota-Baxter pair.
 \end{thm}

 \begin{proof} The proof is divided into the following four steps. For all $h,g\in \mathcal{A}$, we have\\
 Step 1
 \begin{eqnarray*}
 &&\hspace{-15mm}R_{1_{\pi} q}(R_{1_{\pi}}(h1_{\pi})\cdot_{1_{\pi},q} g q +h1_{\pi}\cdot_{1_{\pi},q}R_{q}(g q)+\lambda h1_{\pi} \cdot_{1_{\pi},q} g q )\\
 &=&R_{q}(\mathcal{R}(h)1_{\pi} \cdot_{1_{\pi},q}g q  +h 1_{\pi}\cdot_{1_{\pi},q}\mathcal{R}'(g)q
 +\lambda h1_{\pi}\cdot_{1_{\pi},q} g q)\\
 &=&R_{q}((\mathcal{R}(h)g)q+(h \mathcal{R}'(g))q+\lambda(hg)q)\\
 &=&\mathcal{R}'(\mathcal{R}(h)g+h \mathcal{R}'(g)+\lambda hg)q
 \end{eqnarray*}
 and
 \begin{eqnarray*}
 R_{1_{\pi}}(h 1_{\pi})\cdot_{1_{\pi},q}R_{q}(g q)=(\mathcal{R}(h)1_{\pi})\cdot_{1_{\pi},q}\mathcal{R}'(g)q=(\mathcal{R}(h)\mathcal{R}'(g))q.
 \end{eqnarray*}
 Step 2
 \begin{eqnarray*}
 &&\hspace{-15mm}R_{q 1_{\pi}}(R_{q}(h q)\cdot_{q, 1_{\pi}} g1_{\pi} +h q\cdot_{q,1_{\pi}}R_{1_{\pi}}(g1_{\pi})+\lambda h q \cdot_{q,1_{\pi}} g1_{\pi})\\
 &=&R_{q}(\mathcal{R}'(h)q \cdot_{q,1_{\pi}}g 1_{\pi}  +h q \cdot_{q,1_{\pi}}\mathcal{R}(g)1_{\pi}
 +\lambda h q \cdot_{q,1_{\pi}} g1_{\pi})\\
 &=&R_{q}((\mathcal{R}'(h)g)q+(h \mathcal{R}(g))q+\lambda(hg)q)\\
 &=&\mathcal{R}'(\mathcal{R}'(h)g+h \mathcal{R}(g)+\lambda hg)q
 \end{eqnarray*}
 and
 \begin{eqnarray*}
 R_{q}(h q)\cdot_{q,1_{\pi}}R_{1_{\pi}}(g1_{\pi})
 &=&(\mathcal{R}'(h)q)\cdot_{q,1_{\pi}}\mathcal{R}(g)1_{\pi}=(\mathcal{R}'(h)\mathcal{R}(g))q.
 \end{eqnarray*}
 Step 3
 \begin{eqnarray*}
 &&\hspace{-15mm}R_{1_{\pi} 1_{\pi}}(R_{1_{\pi}}(h1_{\pi})\cdot_{1_{\pi}, 1_{\pi}} g1_{\pi} +h1_{\pi}\cdot_{1_{\pi},1_{\pi}}R_{1_{\pi}}(g1_{\pi})+\lambda h1_{\pi} \cdot_{1_{\pi},1_{\pi}} g1_{\pi})\\
 &=&R_{1_{\pi}}(\mathcal{R}(h)1_{\pi} \cdot_{1_{\pi},1_{\pi}}g 1_{\pi}  +h1_{\pi} \cdot_{1_{\pi},1_{\pi}}\mathcal{R}(g)1_{\pi}
 +\lambda h1_{\pi} \cdot_{1_{\pi},1_{\pi}} g1_{\pi})\\
 &=&R_{1_{\pi}}((\mathcal{R}(h)g)1_{\pi}+(h \mathcal{R}(g))1_{\pi}+\lambda(hg)1_{\pi})\\
 &=&\mathcal{R}(\mathcal{R}(h)g+h \mathcal{R}(g)+\lambda hg)1_{\pi}
 \end{eqnarray*}
 and
 \begin{eqnarray*}
 R_{1_{\pi}}(h 1_{\pi})\cdot_{1_{\pi},1_{\pi}}R_{1_{\pi}}(g1_{\pi})
 =(\mathcal{R}(h)1_{\pi})\cdot_{1_{\pi},1_{\pi}}\mathcal{R}(g)1_{\pi}=(\mathcal{R}(h)\mathcal{R}(g))1_{\pi}.
 \end{eqnarray*}
 Step 4
 \begin{eqnarray*}
 &&\hspace{-15mm}R_{qq}(R_{q}(h q)\cdot_{q,q} g q +h q\cdot_{q,q}R_{q}(g q)+\lambda h q \cdot_{q,q} g q)\\
 &=&R_{q}(\mathcal{R}'(h)q \cdot_{q,q}g q +h q \cdot_{q,q}\mathcal{R}'(g)q
 +\lambda h q \cdot_{q,q} g q)\\
 &=&R_{q}((\mathcal{R}'(h)g)q+(h \mathcal{R}'(g))q+\lambda(hg)q)\\
 &=&\mathcal{R}'(\mathcal{R}'(h)g+h \mathcal{R}'(g)+\lambda hg)q
 \end{eqnarray*}
 and
 \begin{eqnarray*}
 R_{q}(h q)\cdot_{q,q}R_{q}(g q)
 =(\mathcal{R}'(h)q)\cdot_{q,q}\mathcal{R}'(g)q=(\mathcal{R}'(h)\mathcal{R}'(g))q.
 \end{eqnarray*}
 Therefore $(\mathcal{A}[\pi],\{R_{\varphi}\},\lambda)$ is a Rota-Baxter T-algebra of weight $\lambda$ if and only if $(\mathcal{R},\mathcal{R}')$ is a Rota-Baxter pair.
 \end{proof}

 According to Theorem \ref{thm:15.4}, each Rota-Baxter pair determines a Rota-Baxter T-algebra of weight $\lambda$, so by Examples \ref{ex:15.5}-\ref{ex:15.7}, we can get

 \begin{ex}\label{ex:15.8} Let $\pi=\{1,q\}$ be a monoid with a unit $1$ and $q^{2}=q$.

 (1) Let $\mathcal{A}$ be the 2-dimensional algebra defined in Example \ref{ex:15.15}. Then $(\mathcal{A}[\pi],\{R_{\varphi}\},\lambda)$ is a Rota-Baxter T-algebra of weight $\lambda$, where
 \begin{eqnarray*}
 &&R_{1_{\pi}}:\mathcal{A} 1_{\pi}  \longrightarrow \mathcal{A} 1_{\pi},~u_{1} 1_{\pi}  \longmapsto -\lambda u_{1} 1_{\pi},~u_{2} 1_{\pi}  \longmapsto 0 1_{\pi}\\  &&R_{q}:\mathcal{A} q \longrightarrow \mathcal{A} q,~u_{1} q  \longmapsto -\lambda u_{1} q,~u_{2} q  \longmapsto -\lambda u_{2} q.
 \end{eqnarray*}
 Other cases can be given similarly.

 (2) Let $\mathcal{A}$ be the 3-dimensional algebra defined in Example \ref{ex:15.16}. Then $(\mathcal{A}[\pi],\{R_{\varphi}\},\lambda)$ is a Rota-Baxter T-algebra of weight $\lambda$, where
 \begin{eqnarray*}
 &&R_{1_{\pi}}:\mathcal{A} 1_{\pi}  \longrightarrow \mathcal{A} 1_{\pi},~u_{1} 1_{\pi}  \longmapsto -\lambda u_{1} 1_{\pi},~u_{2} 1_{\pi}  \longmapsto 0 1_{\pi},~u_{3} 1_{\pi}  \longmapsto 0 1_{\pi}\\
 &&R_{q}:\mathcal{A} q \longrightarrow \mathcal{A} q,~u_{1} q  \longmapsto -\lambda u_{1} q,~u_{2} q  \longmapsto -\lambda u_{2} q,~u_{3} q  \longmapsto 0 q.
 \end{eqnarray*}
 Other cases are similar.

 (3) Let $\mathcal{A}$ be the 4-dimensional algebra defined in Example \ref{ex:15.17}. Then $(\mathcal{A}[\pi],\{R_{\varphi}\},\lambda)$ is a Rota-Baxter T-algebra of weight $\lambda$, where
 \begin{eqnarray*}
 &&R_{1_{\pi}}:\mathcal{A} 1_{\pi}  \longrightarrow \mathcal{A} 1_{\pi},~u_{1} 1_{\pi}  \longmapsto 0 1_{\pi},~u_{2} 1_{\pi}  \longmapsto 0 1_{\pi},~u_{3} 1_{\pi}  \longmapsto -\lambda u_{3} 1_{\pi},~u_{4} 1_{\pi}  \longmapsto -\lambda u_{4} 1_{\pi}\\
 &&R_{q}:\mathcal{A} q \longrightarrow \mathcal{A} q,~u_{1} q  \longmapsto -\lambda u_{1} q,~u_{2} q  \longmapsto -\lambda u_{2} q,~u_{3} q  \longmapsto -\lambda u_{3} q,~u_{4} q  \longmapsto -\lambda u_{4} q.
 \end{eqnarray*}
 Likewise we can obtain other cases.
 \end{ex}

 \subsection{Rota-Baxter T-algebras, (tri)dendriform T-algebras and T-algebras}
 \subsubsection{Dendriform T-algebras from tridendriform T-algebras }
 Motivated by algebraic K-theory, Loday invented the concept of a dendriform algebra.  Loday and Ronco introduced the concept of a tridendriform algebra (previously also called a dendriform trialgebra) in the study of polytopes and Koszul duality \cite{LR04}.  We now give the T-version.

 \begin{defi}\label{de:5.2} Let $\pi$ be a semigroup. A {\bf dendriform T-algebra} is a family of vector spaces $\{A_{\varphi}\}_{\varphi\in\pi}$ with a family of binary operations $\{\prec_{p,q},\succ_{p,q}: A_{p}\otimes A_{q}\lr A_{pq}\}_{p,q\in\pi}$ such that for $a\in A_{p}$, $b\in A_{q}$, $c\in A_{t}$ and $p,q,t\in \pi$,
 \begin{eqnarray}
 (a\prec_{p,q} b)\prec_{pq,t}c&=&a\prec_{p,qt}(b\prec_{q,t} c+b\succ_{q,t} c),\label{eq:5.4}\\
 (a\succ_{p,q} b)\prec_{pq,t}c&=&a\succ_{p,qt}(b\prec_{q,t} c),\label{eq:5.5}\\
 (a\prec_{p,q} b+a\succ_{p,q} b)\succ_{pq,t}c&=&a\succ_{p,qt}(b\succ_{q,t} c).\label{eq:5.6}
 \end{eqnarray}
 For simplicity, we denote it by $(\{A_{\varphi}\}_{\varphi\in\pi}, \{\prec_{p,q}\}_{p,q\in\pi}, \{\succ_{p,q}\}_{p,q\in\pi})$ or $(A, \prec, \succ)$.
 \end{defi}

 \begin{defi}\label{de:5.3} Let $\pi$ be a semigroup. A {\bf tridendriform T-algebra} is a family of vector spaces $\{A_{\varphi}\}_{\varphi\in\pi}$ with a family of binary operations $\{\prec_{p,q},\succ_{p,q},\bullet_{p,q}: A_{p}\otimes A_{q}\lr A_{pq}\}_{p,q\in\pi}$ such that for $a\in A_{p}$, $b\in A_{q}$, $c\in A_{t}$ and $p,q,t\in \pi$,
 \begin{eqnarray}
 (a\prec_{p,q} b)\prec_{pq,t}c&=&a\prec_{p,qt}(b\prec_{q,t} c+ b\succ_{q,t} c+b\bullet_{q,t} c),\label{eq:5.7}\\
 (a\succ_{p,q} b)\prec_{pq,t}c&=&a\succ_{p,qt}(b\prec_{q,t} c),\label{eq:5.8}\\
 (a\prec_{p,q} b+a\succ_{p,q} b+a\bullet_{p,q} b)\succ_{pq,t}c&=&a\succ_{p,qt}(b\succ_{q,t} c),\label{eq:5.9}\\
 (a\succ_{p,q} b)\bullet_{pq,t}c&=&a\succ_{p,qt}(b\bullet_{q,t} c),\label{eq:5.10}\\
 (a\prec_{p,q} b)\bullet_{pq,t}c&=&a\bullet_{p,qt}(b\succ_{q,t} c),\label{eq:5.11}\\
 (a\bullet_{p,q} b)\prec_{pq,t}c&=&a\bullet_{p,qt}(b\prec_{q,t} c),\label{eq:5.12}\\
 (a\bullet_{p,q} b)\bullet_{pq,t}c&=&a\bullet_{p,qt}(b\bullet_{q,t} c).\label{eq:5.13}
 \end{eqnarray}
 For simplicity, we denote it by $(\{A_{\varphi}\}_{\varphi\in\pi}, \{\prec_{p,q}\}_{p,q\in\pi}, \{\succ_{p,q}\}_{p,q\in\pi},\{\bullet_{p,q}\}_{p,q\in\pi})$ or $(A, \prec, \succ, \bullet)$.
 \end{defi}

 \begin{rmk} (1) If $\pi$ contains a single element (i.e, $\pi$ is trivial), then a (tri)dendriform T-algebra in Definition \ref{de:5.2} (\ref{de:5.3}) is exactly a (tri)dendriform algebra.

 (2) (Tri)dendriform T-algebra is different from (tri)dendriform family algebra introduced in \cite{ZG01}. For example, let $A_{\vp}=A$ in Definition \ref{de:5.2}, this particular case and \cite[Definition 4.1]{ZG01} can not contain each other.
 \end{rmk}

 \begin{pro}\label{pro:9.1} Let $(A, \prec, \succ, \bullet)$ be a tridendriform T-algebra. Then $(\{A_{\varphi}\}_{\varphi\in\pi}, \{\prec'_{p,q}\}_{p,q\in\pi}, \{\succ'_{p,q}\}_{p,q\in\pi})$ is a dendriform T-algebra, where the new operations $\{\prec'_{p,q},\succ'_{p,q}: A_{p}\otimes A_{q}\lr A_{pq}\}_{p,q\in\pi}$ are defined by
 \begin{eqnarray}
 a\prec'_{p,q} b:=a\prec_{p,q} b+a\bullet_{p,q} b,\quad a\succ'_{p,q} b:=a\succ_{p,q} b, \label{eq:9.1}
 \end{eqnarray}
 for $a\in A_{p}$, $b\in A_{q}$ and $p,q\in \pi$.
 \end{pro}

 \begin{proof} We first prove Eq.(\ref{eq:5.4}). For $a\in A_{p}$, $b\in A_{q}$, $c\in A_{t}$ and $p,q,t\in \pi$,
 \begin{eqnarray*}
 (a\prec'_{p,q} b)\prec'_{pq,t} c&\stackrel{(\ref{eq:9.1})}=&
 (a\prec_{p,q} b) \prec_{pq,t} c+ (a\bullet_{p,q} b)\prec_{pq,t} c+(a\prec_{p,q} b) \bullet_{pq,t} c\\
 &&+ (a\bullet_{p,q} b)\bullet_{pq,t} c\\
 &\stackrel{(\ref{eq:5.7})(\ref{eq:5.11})-(\ref{eq:5.13})}=&
 a\prec_{p,qt} (b\prec_{q,t} c+ b\bullet_{q,t} c +b\succ_{q,t} c )+a\bullet_{p,qt} (b\prec_{q,t} c\\
 &&+ b\bullet_{q,t} c +b\succ_{q,t} c)\\
 &\stackrel{(\ref{eq:9.1})}=&a\prec'_{p,qt}(b\prec'_{q,t} c+b\succ'_{q,t} c).
 \end{eqnarray*}
 The rest can be checked in a similar way.
 \end{proof}

 \subsubsection{(Tri)Dendriform T-algebras from Rota-Baxter T-algebras } Rota-Baxter T-algebra of any weight can induce (tri)dendriform algebra as follows.

 \begin{pro}\label{pro:5.4} Let $\pi$ be a semigroup. (1) A Rota-Baxter T-algebra $(A, R)$ induces a tridendriform T-algebra $(\{A_{\varphi}\}_{\varphi\in\pi}, \{\prec_{p,q}\}_{p,q\in\pi}, \{\succ_{p,q}\}_{p,q\in\pi},\{\bullet_{p,q}\}_{p,q\in\pi})$, where
 \begin{eqnarray}
 & a \prec_{p,q} b:= a \cdot_{p,q} R_{q}(b),\quad  a \succ_{p,q} b:= R_{p}(a) \cdot_{p,q} b \quad and \quad a \bullet_{p,q} b:=\lambda a \cdot_{p,q} b, &\label{eq:5.15}
 \end{eqnarray}
 for $a\in A_{p}$, $b\in A_{q}$ and $p,q\in \pi$.

 (2) A Rota-Baxter T-algebra $(A, R)$ induces a dendriform T-algebra $(\{A_{\varphi}\}_{\varphi\in\pi}, \{\prec_{p,q}\}_{p,q\in\pi}, \{\succ_{p,q}\}_{p,q\in\pi})$, where
 \begin{eqnarray}
 &a \prec_{p,q} b:= a \cdot_{p,q} R_{q}(b)+ \lambda a \cdot_{p,q} b,\quad  a \succ_{p,q} b:= R_{p}(a) \cdot_{p,q} b,&\label{eq:5.14}
 \end{eqnarray}
 for $a\in A_{p}$, $b\in A_{q}$ and $p,q\in \pi$.
 \end{pro}

 \begin{proof} (1) We only check that Eqs.(\ref{eq:5.7}), (\ref{eq:5.9}) and (\ref{eq:5.10}) hold. For all $a\in A_p, b\in A_q, c\in A_t$, we calculate
 \begin{eqnarray*}
 (a\prec_{p,q}b)\prec_{pq,t}c
 &\stackrel{(\ref{eq:5.15})}=&(a \cdot_{p,q} R_{q}(b)) \cdot_{pq,t} R_{t}(c)\\
 &\stackrel{(\ref{eq:5.39})}=&a \cdot_{p,qt} R_{qt}(R_{q}(b) \cdot_{q,t} c)+a \cdot_{p,qt} R_{qt}(b \cdot_{q,t} R_{t}(c))+\lambda a \cdot_{p,qt} R_{qt}(b \cdot_{q,t} c)\\
 &\stackrel{(\ref{eq:5.15})}=&
 a\prec_{p,qt}(b\prec_{q,t}c+b\succ_{q,t}c+b\bullet_{q,t}c),
 \end{eqnarray*}
 \begin{eqnarray*}
 (a\prec_{p,q}b+a\succ_{p,q}b+a\bullet_{p,q}b)\succ_{pq,t}c
 &\stackrel{(\ref{eq:5.15})}=&R_{pq}(a \cdot_{p,q} R_{q}(b)+R_{p}(a) \cdot_{p,q} b+\lambda a \cdot_{p,q} b) \cdot_{pq,t} c\\
 &\stackrel{(\ref{eq:5.39})}=&(R_{p}(a)\cdot_{p,q}R_{q}(b))\cdot_{pq,t}c
 \stackrel{(\ref{eq:5.15})}=a\succ_{p,qt}(b\succ_{q,t}c)
 \end{eqnarray*}
 and
 \begin{eqnarray*}
 (a\succ_{p,q}b)\bullet_{pq,t}c
 \stackrel{(\ref{eq:5.39})}=\lambda(R_{p}(a) \cdot_{p,q} b) \cdot_{pq,t} c=\lambda R_{p}(a) \cdot_{p,qt} (b \cdot_{q,t}c)
 \stackrel{(\ref{eq:5.39})}=a\succ_{p,qt}(b\bullet_{q,t}c).
 \end{eqnarray*}
 Other equalities in Definition \ref{de:5.3} can be proved similarly.
 Therefore, $(\{A_{\varphi}\}_{\varphi\in\pi}, \{\prec_{p,q}\}_{p,q\in\pi}, \{\succ_{p,q}\}_{p,q\in\pi},\{\bullet_{p,q}\}_{p,q\in\pi})$ is a tridendriform T-algebra.

 (2) It follows from Part (1) and Proposition \ref{pro:9.1}. This completes the proof.
 \end{proof}

 \begin{rmk} If $\l=0$ and $\pi$ is trivial in Part (2) of Proposition \ref{pro:5.4}, then we obtain \cite[Proposition 4.5]{Ag04}.
 \end{rmk}

 \begin{ex}\label{ex:15.12} By Proposition \ref{pro:5.4} and Example \ref{ex:15.8}, we can obtain a variety of ways to construct (tri)dendriform algebras of $\mathcal{A}[\pi]$ with different dimensions.

 (1) The new structures of tridendriform T-algebra on $\mathcal{A}[\pi]$ (where $\mathcal{A}$ is given in Example \ref{ex:15.15}) can be defined by 
 \begin{eqnarray*}
 &&u_{1}1_{\pi} \prec  u_{1}q=u_{1}1_{\pi} \prec  u_{2}q=u_{1}1_{\pi} \succ  u_{2}q=-u_{1}1_{\pi} \bullet  u_{2}q=u_{2}1_{\pi} \prec  u_{1}q=u_{2}1_{\pi} \succ  u_{1}q\\
 &&=u_{2}1_{\pi} \succ  u_{2}q=-u_{2}1_{\pi} \bullet  u_{2}q=u_{1}q \succ  u_{1}1_{\pi}=-u_{2}1_{\pi} \bullet  u_{1}q=u_{2}1_{\pi} \prec  u_{2}q=u_{1}q \prec  u_{2}1_{\pi}\\
 &&=u_{1}q \succ  u_{2}1_{\pi}=-u_{1}q \bullet  u_{2}1_{\pi}= u_{2}q \prec  u_{1}1_{\pi}=u_{2}q \succ  u_{1}1_{\pi}=-u_{2}q \bullet  u_{1}1_{\pi}=u_{2}q \prec  u_{2}1_{\pi}\\
 &&=u_{2}q \succ  u_{2}1_{\pi}=-u_{2}q \bullet  u_{2}1_{\pi}=u_{1}q \prec  u_{1}q=u_{1}q \succ  u_{1}q=u_{1}q \prec  u_{2}q=u_{1}q \succ  u_{2}q\\
 &&=-u_{1}q \bullet  u_{2}q=u_{2}q \prec  u_{1}q=u_{2}q \succ  u_{1}q=-u_{2}q \bullet  u_{1}q=u_{2}q \prec  u_{2}q=u_{2}q \succ  u_{2}q\\
 &&=-u_{2}q \bullet  u_{2}q=-\lambda u_{2} q,\\
 &&u_{1}1_{\pi} \prec  u_{2}1_{\pi}=u_{1}1_{\pi} \succ  u_{2}1_{\pi}=-u_{1}1_{\pi} \bullet  u_{2}1_{\pi}=u_{2}1_{\pi} \prec  u_{1}1_{\pi}=u_{2}1_{\pi} \succ  u_{1}1_{\pi}\\
 &&=-u_{2}1_{\pi} \bullet  u_{1}1_{\pi}=u_{2}1_{\pi} \prec  u_{2}1_{\pi}=u_{2}1_{\pi} \succ  u_{2}1_{\pi}=-u_{2}1_{\pi} \bullet  u_{2}1_{\pi}=-\lambda u_{2} 1_{\pi},\\
 &&u_{1}1_{\pi} \succ  u_{1}q=-u_{1}1_{\pi} \bullet  u_{1}q=u_{1}q \prec  u_{1}1_{\pi}=-u_{1}q \bullet  u_{1}q=-u_{1}q \bullet  u_{1}1_{\pi}=-\lambda u_{1} q,\\
 &&u_{1}1_{\pi} \prec  u_{1}1_{\pi}=u_{1}1_{\pi} \succ  u_{1}1_{\pi}=-u_{1}1_{\pi} \bullet  u_{1}1_{\pi}=-\lambda u_{1} 1_{\pi}.
 \end{eqnarray*}

 (2) The new structures of tridendriform T-algebra on $\mathcal{A}[\pi]$ (where $\mathcal{A}$ is given in Example \ref{ex:15.17}) can be defined by 
 \begin{eqnarray*}
 &&u_{1}1_{\pi} \prec  u_{1}q=u_{1}1_{\pi} \succ  u_{1}q=-u_{1}1_{\pi} \bullet  u_{1}q=u_{2}1_{\pi} \prec  u_{2}q=-u_{2}1_{\pi} \bullet  u_{2}q=u_{1}q \prec  u_{1}1_{\pi}\\
 &&=u_{1}q \succ  u_{1}1_{\pi}= -u_{1}q \bullet  u_{1}1_{\pi}=u_{2}q \succ  u_{2}1_{\pi}=-u_{2}q \bullet  u_{2}1_{\pi}=u_{1}q \prec  u_{1}q=u_{1}q \succ  u_{1}q\\
 &&=-u_{1}q \bullet  u_{1}q=u_{2}q \prec  u_{2}q=u_{2}q \succ  u_{2}q=-u_{2}q \bullet  u_{2}q=-\lambda u_{1} q,\\
 &&u_{1}1_{\pi} \prec  u_{2}q=u_{1}1_{\pi} \succ  u_{2}q=-u_{1}1_{\pi} \bullet  u_{2}q=u_{2}1_{\pi} \prec  u_{1}q=-u_{2}1_{\pi} \bullet  u_{1}q=u_{1}q \succ  u_{2}1_{\pi}\\
 &&= -u_{1}q \bullet  u_{2}1_{\pi}=u_{2}q \prec  u_{1}1_{\pi}=u_{2}q \succ  u_{1}1_{\pi}=-u_{2}q \bullet  u_{1}1_{\pi}=u_{1}q \prec  u_{2}q=u_{1}q \succ  u_{2}q\\
 &&=-u_{1}q \bullet  u_{2}q=u_{2}q \prec  u_{1}q=u_{2}q \succ  u_{1}q=-u_{2}q \bullet  u_{1}q=-\lambda u_{2} q,\\
 &&u_{1}1_{\pi} \prec  u_{3}q=u_{1}1_{\pi} \succ  u_{3}q=-u_{1}1_{\pi} \bullet  u_{3}q=u_{2}1_{\pi} \prec  u_{4}q=-u_{2}1_{\pi} \succ  u_{4}q=-u_{2}1_{\pi} \bullet  u_{4}q\\
 &&=u_{3}1_{\pi} \prec  u_{1}q=-u_{3}1_{\pi} \bullet  u_{1}q=-u_{4}1_{\pi} \prec  u_{2}q=u_{4}1_{\pi} \bullet  u_{2}q=u_{1}q \succ  u_{3}1_{\pi}=-u_{1}q \bullet  u_{3}1_{\pi}\\
 &&=u_{2}q \succ  u_{4}1_{\pi}=-u_{2}q \bullet  u_{4}1_{\pi}=u_{3}q \prec  u_{1}1_{\pi}=u_{3}q \succ  u_{1}1_{\pi}=u_{3}q \bullet  u_{1}1_{\pi}=u_{4}q \prec  u_{2}1_{\pi}\\
 &&=-u_{4}q \succ  u_{2}1_{\pi}=u_{4}q \bullet  u_{2}1_{\pi}=u_{1}q \prec  u_{3}q=u_{1}q \succ  u_{3}q=-u_{1}q \bullet  u_{3}q=u_{2}q \prec  u_{4}q\\
 &&=u_{2}q \succ  u_{4}q=-u_{2}q \bullet  u_{4}q=u_{3}q \prec  u_{1}q=u_{3}q \succ  u_{1}q=-u_{3}q \bullet  u_{1}q=-u_{4}q \prec  u_{2}q\\
 &&=-u_{4}q \succ  u_{2}q=u_{4}q \bullet  u_{2}q=-\lambda u_{3} q,\\
 &&u_{1}1_{\pi} \prec  u_{4}q=u_{1}1_{\pi} \succ  u_{4}q=-u_{1}1_{\pi} \bullet  u_{4}q=u_{2}1_{\pi} \prec  u_{3}q=-u_{2}1_{\pi} \succ  u_{3}q=-u_{2}1_{\pi} \bullet  u_{3}q\\
 &&=-u_{3}1_{\pi} \prec  u_{2}q=u_{4}1_{\pi} \prec  u_{1}q=u_{3}1_{\pi} \bullet  u_{2}q=-u_{4}1_{\pi} \bullet  u_{1}q=u_{1}q \succ  u_{4}1_{\pi}=-u_{1}q \bullet  u_{4}1_{\pi}\\
 &&=u_{2}q \succ  u_{3}1_{\pi}=-u_{2}q \bullet  u_{3}1_{\pi}=u_{3}q \prec  u_{2}1_{\pi}=-u_{3}q \succ  u_{2}1_{\pi}=u_{3}q \bullet  u_{2}1_{\pi}=u_{4}q \prec  u_{1}1_{\pi}\\
 &&=u_{4}q \succ  u_{1}1_{\pi}=-u_{4}q \bullet  u_{1}1_{\pi}=u_{1}q \prec  u_{4}q=u_{1}q \succ  u_{4}q=-u_{1}q \bullet  u_{4}q=u_{2}q \prec  u_{3}q\\
 &&=u_{2}q \succ  u_{3}q=-u_{2}q \bullet  u_{3}q=-u_{3}q \prec  u_{2}q=-u_{3}q \succ  u_{2}q=u_{3}q \bullet  u_{2}q=u_{4}q \prec  u_{1}q\\
 &&=u_{4}q \succ  u_{1}q=-u_{4}q \bullet  u_{1}q=-\lambda u_{4} q,\\
 &&u_{2}1_{\pi} \succ  u_{1}q= \lambda u_{2} q+p_{1}u_{3}q+\frac{p_{1}p_{2}}{\lambda+p_{2}}u_{4}q,\quad u_{2}1_{\pi} \succ  u_{2}q=\lambda u_{1} q-p_{3}u_{4}q-\frac{p_{1}p_{2}}{\lambda+p_{2}}u_{3}q,\\
 &&u_{3}1_{\pi} \succ  u_{1}q=-(\lambda+p_{2})u_{3}q-p_{2}u_{4}q,\quad u_{4}1_{\pi} \succ  u_{2}q=-(\lambda+p_{2})u_{4}q-p_{2}u_{3}q,\\
 &&u_{3}1_{\pi} \succ  u_{2}q=(\lambda+p_{2})u_{4}q+p_{2}u_{3}q,\quad
   u_{4}1_{\pi} \succ_{1_{\pi},q} u_{1}q= (\lambda+p_{2})u_{3}q+p_{2}u_{4}q,\\
 &&u_{1}q \prec  u_{2}1_{\pi}=\lambda u_{2} q+p_{1}u_{3}q+\frac{p_{1}p_{2}}{\lambda+p_{2}}u_{4}q,\quad u_{1}q \prec  u_{3}1_{\pi}=-(\lambda+p_{2})u_{3}q-p_{2}u_{4}q,\\
 &&u_{1}q \prec  u_{4}1_{\pi}= (\lambda+p_{2})u_{3}q+p_{2}u_{4}q,\quad u_{2}q \prec  u_{2}1_{\pi}= \lambda u_{1} q+p_{1}u_{4}q+\frac{p_{1}p_{2}}{\lambda+p_{2}}u_{3}q,\\
 &&u_{2}q \prec  u_{3}1_{\pi}=-(\lambda+p_{2})u_{4}q-p_{2}u_{3}q,\quad u_{2}q \prec  u_{4}1_{\pi}=(\lambda+p_{2})u_{4}q+p_{2}u_{3}q,\\
 &&u_{1}1_{\pi} \prec  u_{2}1_{\pi}= \lambda u_{2} 1_{\pi}+p_{1}u_{3}1_{\pi}+\frac{p_{1}p_{2}}{\lambda+p_{2}}u_{4}1_{\pi}=u_{2}1_{\pi} \succ  u_{1}1_{\pi},\\
 &&u_{1}1_{\pi} \prec  u_{3}1_{\pi}= -(\lambda+p_{2})u_{3}1_{\pi}-p_{2}u_{4}1_{\pi}=u_{3}1_{\pi} \succ  u_{1}1_{\pi},\\
 &&u_{1}1_{\pi} \prec  u_{4}1_{\pi}= (\lambda+p_{2})u_{3}1_{\pi}+p_{2}u_{4}1_{\pi}=u_{4}1_{\pi} \succ  u_{1}1_{\pi},\\
 &&u_{2}1_{\pi} \prec  u_{2}1_{\pi}= \lambda u_{1} 1_{\pi}+p_{1}u_{4}1_{\pi}+\frac{p_{1}p_{2}}{\lambda+p_{2}}u_{3}1_{\pi}, ~u_{2}1_{\pi} \succ  u_{2}1_{\pi}=\lambda u_{1} 1_{\pi}-p_{1}u_{4}1_{\pi}-\frac{p_{1}p_{2}}{\lambda+p_{2}}u_{3}1_{\pi},\\
 &&u_{2}1_{\pi} \prec  u_{3}1_{\pi}=-(\lambda+p_{2})u_{4}1_{\pi}+p_{2}u_{3}1_{\pi},\quad u_{2}1_{\pi} \prec  u_{4}1_{\pi}= (\lambda+p_{2})u_{4}1_{\pi}+p_{2}u_{3}1_{\pi},\\
 &&u_{1}1_{\pi} \prec  u_{1}1_{\pi}=u_{1}1_{\pi} \succ  u_{1}1_{\pi}=-u_{1}1_{\pi} \bullet  u_{1}1_{\pi}=-u_{2}1_{\pi} \bullet  u_{2}1_{\pi}=-u_{1}1_{\pi} \bullet  u_{1}1_{\pi}=-\lambda u_{1} 1_{\pi},\\
 &&u_{1}1_{\pi} \succ  u_{2}1_{\pi}=-u_{1}1_{\pi} \bullet  u_{2}1_{\pi}=u_{2}1_{\pi} \prec  u_{1}1_{\pi}= -u_{2}1_{\pi} \bullet  u_{1}1_{\pi}=-\lambda u_{2} 1_{\pi},\\
 &&u_{1}1_{\pi} \succ  u_{3}1_{\pi}=-u_{1}1_{\pi} \bullet  u_{3}1_{\pi}=u_{3}1_{\pi} \prec  u_{1}1_{\pi}=-u_{2}1_{\pi} \succ  u_{4}1_{\pi}=-u_{2}1_{\pi} \bullet  u_{4}1_{\pi}=-u_{3}1_{\pi} \bullet  u_{1}1_{\pi}\\
 &&=u_{4}1_{\pi} \prec  u_{2}1_{\pi}=u_{4}1_{\pi} \bullet  u_{2}1_{\pi}=-\lambda u_{3} 1_{\pi},\\
 &&u_{1}1_{\pi} \succ  u_{4}1_{\pi}=-u_{1}1_{\pi} \bullet  u_{4}1_{\pi}=-u_{2}1_{\pi} \succ  u_{3}1_{\pi}=-u_{2}1_{\pi} \bullet  u_{3}1_{\pi}=u_{3}1_{\pi} \prec  u_{2}1_{\pi}=u_{3}1_{\pi} \bullet u_{2}1_{\pi}\\
 &&=u_{4}1_{\pi} \prec  u_{1}1_{\pi}=-u_{4}1_{\pi} \bullet  u_{1}1_{\pi}=-\lambda u_{4} 1_{\pi},\\
 &&u_{3}1_{\pi} \succ  u_{2}1_{\pi}=(\lambda+p_{2})u_{4}1_{\pi}+p_{2}u_{3}1_{\pi},\quad u_{4}1_{\pi} \succ  u_{2}1_{\pi}=-(\lambda+p_{2})u_{4}1_{\pi}-p_{2}u_{3}1_{\pi}.
 \end{eqnarray*}
 The operations for the remaining cases are 0.

 (3) Based on Proposition \ref{pro:9.1} and (1), (2) above, we can get new structures of dendriform T-algebra on $\mathcal{A}[\pi]$, where $\mathcal{A}$ are given in Example \ref{ex:15.15} and Example \ref{ex:15.17}, respectively.
 \end{ex}

 \subsubsection{T-algebras from (tri)dendriform T-algebras}

 \begin{pro}\label{pro:5.5} (1) Let $(A, \prec, \succ)$ be a dendriform T-algebra. Then $(\{A_{\varphi}\}_{\varphi\in\pi},$ $\{\diamond_{p,q}\}_{p,q\in\pi})$ is a T-algebra, where $\{\diamond_{p,q}:A_{p}\otimes A_{q}\lr A_{pq}\}_{p,q\in\pi}$,
 \begin{eqnarray}
 &a \diamond_{p,q} b := a \prec_{p,q}b+a \succ_{p,q}b,\quad for\  a\in A_{p}, b\in A_{q}.&\label{eq:5.16}
 \end{eqnarray}

 (2) Let $(A, \prec, \succ, \bullet)$ be a tridendriform T-algebra. Then $(\{A_{\varphi}\}_{\varphi\in\pi},$ $\{\diamond_{p,q}\}_{p,q\in\pi})$ is a T-algebra, where $\{\diamond_{p,q}:A_{p}\otimes A_{q}\lr A_{pq}\}_{p,q\in\pi}$,
 \begin{eqnarray}
 &a \diamond_{p,q} b := a \prec_{p,q}b+a \bullet_{p,q}b+a \succ_{p,q}b,\quad for\  a\in A_{p},b\in A_{q}.&\label{eq:5.17}
 \end{eqnarray}
 \end{pro}

 \begin{proof} (1) For all $a\in A_p, b\in A_q, c\in A_t$, we get
 \begin{eqnarray*}
 (a\diamond_{pq}b)\diamond_{pq,t}c
  &\stackrel{(\ref{eq:5.16})}=&(a\prec_{p,q}b+a\succ_{p,q}b)\prec_{pq,t}c
  +(a\prec_{p,q}b+a\succ_{p,q}b)\succ_{pq,t}c\\
  &\stackrel{(\ref{eq:5.4})(\ref{eq:5.5})(\ref{eq:5.6})}=&a\prec_{p,qt}(b\prec_{q,t}c+b\succ_{q,t}c)
  +a\succ_{p,qt}(b\prec_{q,t}c)+a\succ_{p,qt}(b\succ_{q,t}c)\\
  &\stackrel{(\ref{eq:5.16})}=&a\diamond_{p,qt}(b\diamond_{q,t}c).
 \end{eqnarray*}
 Thus $(\{A_{\varphi}\}_{\varphi\in\pi},\{\diamond_{p,q}\}_{p,q\in\pi})$ is a T-algebra.

 (2) can be obtained by Part (1) and Proposition \ref{pro:9.1}.
 \end{proof}

 \begin{rmk}\label{rmk:15.19} We can get many new T-algebra structures on $\mathcal{A}[\pi]$ by Proposition \ref{pro:5.5} and Example \ref{ex:15.12}.
 \end{rmk}

 \begin{ex}\label{ex:15.19a} (1) In general, according to Part (1) in Proposition \ref{pro:5.5}, commutative dendriform T-algebras induce commutative T-algebras. The following example comes from (3) in Example \ref{ex:15.12} which shows that noncommutative dendriform T-algebra can also induce commutative T-algebra.

 The new structure of T-algebra on $\mathcal{A}[\pi]$ (where $\mathcal{A}$ is given in Example \ref{ex:15.15}) can be defined by
 \begin{eqnarray*}
 &&u_{1}1_{\pi} \diamond  u_{1}q=u_{1}q \diamond  u_{1}1_{\pi}=u_{1}1_{\pi} \diamond  u_{2}q=u_{2}q \diamond  u_{1}1_{\pi} =u_{2}1_{\pi} \diamond  u_{1}q= u_{1}q \diamond  u_{2}1_{\pi}=u_{2}1_{\pi} \diamond  u_{2}q\\
 &&=u_{2}q \diamond  u_{2}1_{\pi}=u_{2}q \diamond  u_{2}q=u_{1}q \diamond  u_{2}q=u_{2}q \diamond  u_{1}q =-\lambda u_{2} q,\quad u_{1}q \diamond  u_{1}q=\lambda u_{1} q-2\lambda u_{2} q,\\
 &&u_{1}1_{\pi} \diamond  u_{1}1_{\pi}=-\lambda u_{1} 1_{\pi},~
   u_{1}1_{\pi} \diamond  u_{2}1_{\pi}=u_{2}1_{\pi} \diamond  u_{1}1_{\pi}=u_{2}1_{\pi} \diamond  u_{2}1_{\pi}=-\lambda u_{2} 1_{\pi}.
 \end{eqnarray*}

 (2) The following example comes from (4) in Example \ref{ex:15.12} which shows that noncommutative dendriform algebra can also induce noncommutative T-algebra.

 The new structure of T-algebra on $\mathcal{A}[\pi]$ (where $\mathcal{A}$ is given in Example \ref{ex:15.17}) can be defined by
 \begin{eqnarray*}
 &&u_{1}1_{\pi} \diamond  u_{1}q=u_{1}q \diamond  u_{1}1_{\pi}=u_{1}q \diamond  u_{1}q=u_{2}q \diamond  u_{2}q=-\lambda u_{1} q,\quad u_{1}1_{\pi} \diamond  u_{2}q=u_{2}q \diamond  u_{1}1_{\pi}\\
 &&=u_{1}q \diamond  u_{2}q=u_{2}q \diamond  u_{1}q= -\lambda u_{2} q,\quad u_{1}1_{\pi} \diamond  u_{3}q=u_{3}q \diamond  u_{1}1_{\pi}=-u_{2}1_{\pi} \diamond  u_{4}q=u_{4}q \diamond  u_{2}1_{\pi}\\
 &&=u_{1}q \diamond  u_{3}q=u_{3}q \diamond  u_{1}q=u_{2}q \diamond  u_{4}q=-u_{4}q \diamond  u_{2}q=-\lambda u_{3} q,\quad u_{1}1_{\pi} \diamond  u_{4}q=u_{4}q \diamond  u_{1}1_{\pi}\\
 &&=-u_{2}1_{\pi} \diamond  u_{3}q=u_{3}q \diamond  u_{2}1_{\pi}=u_{1}q \diamond  u_{4}q=u_{4}q \diamond  u_{1}q=u_{2}q \diamond  u_{3}q=-u_{3}q \diamond  u_{2}q=-\lambda u_{4} q,\\
 &&u_{2}1_{\pi} \diamond  u_{1}q=u_{1}q \diamond  u_{2}1_{\pi}= \lambda u_{2} q+p_{1}u_{3}q+\frac{p_{1}p_{2}}{\lambda+p_{2}}u_{4}q,\  u_{2}1_{\pi} \diamond  u_{2}q= \lambda u_{1} q-p_{3}u_{4}q-\frac{p_{1}p_{2}}{\lambda+p_{2}}u_{3}q,\\
 &&u_{3}1_{\pi} \diamond  u_{1}q=u_{1}q \diamond  u_{3}1_{\pi}= -(\lambda+p_{2})u_{3}q-p_{2}u_{4}q,\quad u_{2}q \diamond  u_{2}1_{\pi}= \lambda u_{1} q+p_{1}u_{4}q+\frac{p_{1}p_{2}}{\lambda+p_{2}}u_{3}q,\\
 &&u_{3}1_{\pi} \diamond  u_{2}q=-u_{2}q \diamond  u_{3}1_{\pi}= (\lambda+p_{2})u_{4}q+p_{2}u_{3}q,\qquad\  u_{1}1_{\pi} \diamond  u_{1}1_{\pi}= -\lambda u_{1} 1_{\pi},\\
 &&u_{4}1_{\pi} \diamond  u_{1}q=u_{1}q \diamond  u_{4}1_{\pi}= (\lambda+p_{2})u_{3}q+p_{2}u_{4}q,\qquad\quad u_{2}1_{\pi} \diamond  u_{2}1_{\pi}= 3\lambda u_{1} 1_{\pi},\\
 &&-u_{4}1_{\pi} \diamond  u_{2}q=u_{2}q \diamond  u_{4}1_{\pi}=  (\lambda+p_{2})u_{4}q+p_{2}u_{3}q,\qquad u_{2}1_{\pi} \diamond  u_{3}1_{\pi}= (\lambda-p_{2})u_{4}1_{\pi}+p_{2}u_{3}1_{\pi},\\
 &&u_{1}1_{\pi} \diamond  u_{2}1_{\pi}=u_{2}1_{\pi} \diamond  u_{1}1_{\pi}= \lambda u_{2} 1_{\pi}+p_{1}u_{3}1_{\pi}+\frac{p_{1}p_{2}}{\lambda+p_{2}}u_{4}1_{\pi},\\
 &&u_{1}1_{\pi} \diamond  u_{3}1_{\pi}=u_{3}1_{\pi} \diamond  u_{1}1_{\pi}= -(\lambda+p_{2})u_{3}1_{\pi}-p_{2}u_{4}1_{\pi},\\
 &&u_{2}1_{\pi} \diamond  u_{4}1_{\pi}=-u_{4}1_{\pi} \diamond  u_{2}1_{\pi}= (\lambda+p_{2})u_{4}1_{\pi}+(2\lambda+p_{2})u_{3}1_{\pi},\\
 &&u_{4}1_{\pi} \diamond  u_{1}1_{\pi}=u_{1}1_{\pi} \diamond  u_{4}1_{\pi}=(\lambda+p_{2})u_{3}1_{\pi}+p_{2}u_{4}1_{\pi},\quad u_{3}1_{\pi} \diamond  u_{2}1_{\pi}= (-\lambda+p_{2})u_{4}1_{\pi}+p_{2}u_{3}1_{\pi}.
 \end{eqnarray*}
 The operations for the remaining cases are 0.
 \end{ex}

 \begin{cor}\label{cor:9.3} Let $(A, R)$ be a Rota-Baxter T-algebra. If we define a new multiplication $\{\diamond_{p,q}:A_{p}\otimes A_{q}\lr A_{pq}\}_{p,q\in\pi}$ by
 \begin{eqnarray}
 &a \diamond_{p,q} b:=a \cdot_{p,q} R(b)+ R(a) \cdot_{p,q} b+\lambda a \cdot_{p,q} b,&\label{eq:9.3}
 \end{eqnarray}
 for $a\in A_{p}$, $b\in A_{q}$ and $p,q\in\pi$. Then $(\{A_{\varphi}\}_{\varphi\in \pi}, \{\diamond_{p,q}\}_{p,q\in \pi})$ is a T-algebra.
 \end{cor}

 \begin{proof}
 It follows from Part (2) in Proposition \ref{pro:5.5} and Part (2) in Proposition \ref{pro:5.4}.
 \end{proof}

 \begin{rmk} For a Rota-Baxter T-algebra $(A, R)$, we have $R_{p}(a)\cdot_{p,q} R_{q}(b)=R_{pq}( a\diamond_{p,q} b)$.
 \end{rmk}

 \section{Rota-Baxter Hopf T-(co)algebras}\label{se:rbhta} In this section, we introduce a class of bialgebraic structures (named Rota-Baxter Hopf T-(co)algebra) on Rota-Baxter T-(co)algebras.

 \subsection{Rota-Baxter Hopf T-algebras}
 \begin{defi}\label{de:5.16} Let $\pi$ be a semigroup. A Rota-Baxter T-algebra $(A, R)$ is a {\bf Rota-Baxter semi-Hopf T-algebra} denoted by $(A, \D, R)$ if every $\{A_{\varphi}\}_{\varphi\in\pi}$ is a coalgebra with comultiplication $\D_{\varphi}$ such that $\{\mu_{p,q}\}_{p,q\in\pi}$ are coalgebra maps and $R_{\varphi}$ is a Rota-Baxter cooperator (see \cite{E06,ML}), i.e.,
 \begin{eqnarray}
 &(R_{\varphi}\otimes R_{\varphi})\Delta_{\varphi}(a)=(R_{\varphi}\otimes \id_{A_\varphi})\Delta_{\varphi} R_{\varphi}(a)+(\id_{A_\varphi}\otimes R_{\varphi})\Delta_{\varphi} R_{\varphi}(a)+\lambda \Delta_{\varphi} R_{\varphi}(a),&\label{eq:5.40}
 \end{eqnarray}
 for all $a\in A_\varphi, \varphi\in \pi$. If, furthermore, $\pi$ is a monoid with unit 1, and Rota-Baxter T-algebra is unital such that $\eta: K\lr A_1$ is a coalgebra map, then we call Rota-Baxter semi-Hopf T-algebra {\bf unital}.

 Let $\pi$ be a group. A {\bf Rota-Baxter Hopf T-algebra} denoted by $(A, \D, R, S)$  is a unital Rota-Baxter semi-Hopf T-algebra together with a family of linear maps $\{S_{\varphi}:A_{\varphi}\lr A_{\varphi^{-1}}\}_{\varphi\in\pi}$ such that
 \begin{eqnarray}
 &\mu_{\varphi^{-1},\varphi}\circ (S_{\varphi}\otimes \id_{A_\varphi})\circ \Delta_{\varphi}=\eta \varepsilon_{\varphi}=\mu_{\varphi,\varphi^{-1}}\circ (\id_{A_\varphi}\otimes S_{\varphi})\circ \Delta_{\varphi},&\label{eq:5.41}\\
 &S_{\varphi}\circ R_{\varphi}=R_{\varphi^{-1}}\circ S_{\varphi}.&\label{eq:5.42}
 \end{eqnarray}
 \end{defi}

 \begin{rmk}\label{rmk:5.17} (1) Rota-Baxter Hopf T-algebra $(A, \D, R, S)$ includes a Hopf T-algebra $(A, \D, S)$ in \cite{Tu,Tu1}.

 (2) If $(\mathcal{A},\widetilde{\mu}, \widetilde{\eta},\widetilde{\Delta}, \widetilde{\varepsilon},\mathcal{S})$ is a Hopf algebra, then we call Rota-Baxter Hopf T-algebra $(\{A_{\varphi}=\mathcal{A}\}, \{\mu_{p,q}=\widetilde{\mu}\}, \{R_{\varphi}\}, \lambda, \{\Delta_{\varphi}=\widetilde{\Delta}\}, \{\varepsilon_{\varphi}=\widetilde{\varepsilon}\}, \{S_{\varphi}=\mathcal{S}\})$ a {\bf Rota-Baxter family Hopf algebra of weight $\lambda$}.

 (3) If $\pi=\{1\}$, then we call Rota-Baxter Hopf T-algebra $(A_{1},\mu_{1,1},R_{1}, \lambda, \Delta_{1}, \varepsilon_{1}, S_{1})$ is a {\bf $(R_1,R_1)$-Rota-Baxter Hopf algebra of weight $\l$}. We note here $(R_1, R_1)$-Rota-Baxter Hopf algebra of weight $\l$ is a $(R_1,R_1)$-Rota-Baxter bialgebra of weight $\l$ introduced in \cite{ML} together with antipode $S_1$ such that $S_1\circ R_1=R_1\circ S_1$.
 \end{rmk}

 \begin{ex}\label{ex:15.11} Let $\pi=\{1,q\}$ be a monoid with a unit $1$ and $q^{2}=q$.

 (1) Let $\mathcal{A}$ be the 2-dimensional algebra defined in Example \ref{ex:15.15}. For all $\varphi\in \pi$, define $\D_{\varphi}: \mathcal{A}\varphi\lr \mathcal{A}\varphi$ by $\D(h\varphi)=h\varphi\o h\varphi$, $\forall~h\in \mathcal{A}$. Then $(\mathcal{A}[\pi], \D, R)$ is a Rota-Baxter semi-Hopf T-algebra of weight $\lambda$ with $\{R_{\varphi}\}$ given by
 \begin{eqnarray*}
 &&R_{1_{\pi}}:\mathcal{A} 1_{\pi}  \longrightarrow \mathcal{A} 1_{\pi},~u_1 1_{\pi}  \longmapsto -\lambda u_1 1_{\pi},~u_2 1_{\pi}  \longmapsto 0 1_{\pi}\\ &&R_{q}:\mathcal{A} q \longrightarrow \mathcal{A} q, ~u_1 q  \longmapsto -\lambda u_1 q,~u_2 q  \longmapsto -\lambda u_2 q;
 \end{eqnarray*}
 \begin{eqnarray*}
 &&R_{1_{\pi}}:\mathcal{A} 1_{\pi}  \longrightarrow \mathcal{A} 1_{\pi},~u_1 1_{\pi}  \longmapsto 0 1_{\pi},~u_2 1_{\pi}  \longmapsto -\lambda u_2 1_{\pi}\\  &&R_{q}:\mathcal{A} q \longrightarrow \mathcal{A} q,~u_1 q  \longmapsto -\lambda u_1 q,~u_2 q  \longmapsto -\lambda u_2 q;
 \end{eqnarray*}
 \begin{eqnarray*}
 &&\quad R_{1_{\pi}}:\mathcal{A} 1_{\pi}  \longrightarrow \mathcal{A} 1_{\pi},~ u_1 1_{\pi}  \longmapsto -\lambda u_1 1_{\pi},~u_2 1_{\pi}  \longmapsto -\lambda u_2 1_{\pi}\\
 &&\quad R_{q}:\mathcal{A} q \longrightarrow \mathcal{A} q,~ u_1 q  \longmapsto 0 q,~u_2 q  \longmapsto -\lambda u_2 q.
 \end{eqnarray*}

 (2) Let $\mathcal{A}$ be the 3-dimensional algebra defined in Example \ref{ex:15.16}. For all $\varphi\in \pi$, define $\D_{\varphi}: \mathcal{A}\varphi\lr \mathcal{A}\varphi$ by $\D(h\varphi)=h\varphi\o h\varphi$, $\forall~h\in \mathcal{A}$. Then $(\mathcal{A}[\pi], \D, R)$ is a Rota-Baxter semi-Hopf T-algebra of weight $\lambda$ with $\{R_{\varphi}\}$ given by
 \begin{eqnarray*}
 &&R_{1_{\pi}}:\mathcal{A} 1_{\pi}  \longrightarrow \mathcal{A} 1_{\pi},~ u_1 1_{\pi}  \longmapsto -\lambda u_1 1_{\pi},~u_2 1_{\pi}  \longmapsto 0 1_{\pi},~u_3 1_{\pi}  \longmapsto 0 1_{\pi}\\
 &&R_{q}:\mathcal{A} q \longrightarrow \mathcal{A} q,~u_1 q  \longmapsto -\lambda u_1 q,~u_2 q  \longmapsto -\lambda u_2 q,~u_3 q  \longmapsto 0 q;
 \end{eqnarray*}
 \begin{eqnarray*}
 &&R_{1_{\pi}}:\mathcal{A} 1_{\pi}  \longrightarrow \mathcal{A} 1_{\pi},~u_1 1_{\pi}  \longmapsto -\lambda u_1 1_{\pi},~u_2 1_{\pi}  \longmapsto 0 1_{\pi},~u_3 1_{\pi}  \longmapsto 0 1_{\pi}\\
 &&R_{q}:\mathcal{A} q \longrightarrow \mathcal{A} q,~u_1 q  \longmapsto -\lambda u_1 q,,~u_2 q  \longmapsto -\lambda u_2 q,,~u_3 q  \longmapsto -\lambda u_3 q;
 \end{eqnarray*}
 \begin{eqnarray*}
 &&R_{1_{\pi}}:\mathcal{A} 1_{\pi}  \longrightarrow \mathcal{A} 1_{\pi},~u_1 1_{\pi}  \longmapsto 0 1_{\pi},~u_2 1_{\pi}  \longmapsto -\lambda u_2 1_{\pi},~u_3 1_{\pi}  \longmapsto 0 1_{\pi}\\
 &&R_{q}:\mathcal{A} q \longrightarrow \mathcal{A} q,~u_1 q  \longmapsto -\lambda u_1 q,~u_2 q  \longmapsto -\lambda u_2 q,~u_3 q  \longmapsto 0 q;
 \end{eqnarray*}
 \begin{eqnarray*}
 &&R_{1_{\pi}}:\mathcal{A} 1_{\pi}  \longrightarrow \mathcal{A} 1_{\pi},~u_1 1_{\pi}  \longmapsto 0 1_{\pi},~u_2 1_{\pi}  \longmapsto -\lambda u_2 1_{\pi},~u_3 1_{\pi}  \longmapsto 0 1_{\pi} \\
 &&R_{q}:\mathcal{A} q \longrightarrow \mathcal{A} q,~u_1 q  \longmapsto -\lambda u_1 q,~u_2 q  \longmapsto -\lambda u_2 q,~u_3 q  \longmapsto -\lambda u_3 q;
 \end{eqnarray*}
 \begin{eqnarray*}
 &&R_{1_{\pi}}:\mathcal{A} 1_{\pi}  \longrightarrow \mathcal{A} 1_{\pi},~u_1 1_{\pi}  \longmapsto 0 1_{\pi},~u_2 1_{\pi}  \longmapsto 0 1_{\pi},~u_3 1_{\pi}  \longmapsto -\lambda u_3 1_{\pi}\\
 &&R_{q}:\mathcal{A} q \longrightarrow \mathcal{A} q,~u_1 q  \longmapsto -\lambda u_1 q,~u_2 q  \longmapsto 0 q,~u_3 q  \longmapsto -\lambda u_3 q;
 \end{eqnarray*}
 \begin{eqnarray*}
 &&R_{1_{\pi}}:\mathcal{A} 1_{\pi}  \longrightarrow \mathcal{A} 1_{\pi},~u_1 1_{\pi}  \longmapsto 0 1_{\pi},~u_2 1_{\pi}  \longmapsto 0 1_{\pi},~u_3 1_{\pi}  \longmapsto -\lambda u_3 1_{\pi}\\
 &&R_{q}:\mathcal{A} q \longrightarrow \mathcal{A} q,~u_1 q  \longmapsto -\lambda u_1 q,~u_2 q  \longmapsto -\lambda u_2 q,~u_3 q  \longmapsto -\lambda u_3 q;
 \end{eqnarray*}
 \begin{eqnarray*}
 &&R_{1_{\pi}}:\mathcal{A} 1_{\pi}  \longrightarrow \mathcal{A} 1_{\pi},~u_1 1_{\pi}  \longmapsto -\lambda u_1 1_{\pi},~u_2 1_{\pi}  \longmapsto -\lambda u_2 1_{\pi},~u_3 1_{\pi}  \longmapsto 0 1_{\pi}\\
 &&R_{q}:\mathcal{A} q \longrightarrow \mathcal{A} q,~u_1 q  \longmapsto -\lambda u_1 q,~u_2 q  \longmapsto -\lambda u_2 q,~u_3 q  \longmapsto -\lambda u_3 q;
 \end{eqnarray*}
 \begin{eqnarray*}
 &&R_{1_{\pi}}:\mathcal{A} 1_{\pi}  \longrightarrow \mathcal{A} 1_{\pi},~u_1 1_{\pi}  \longmapsto 0 1_{\pi},~u_2 1_{\pi}  \longmapsto -\lambda u_2 1_{\pi},~u_3 1_{\pi}  \longmapsto -\lambda u_3 1_{\pi}\\
 &&R_{q}:\mathcal{A} q \longrightarrow \mathcal{A} q,~u_1 q  \longmapsto -\lambda u_1 q,~u_2 q  \longmapsto -\lambda u_2 q,~u_3 q  \longmapsto -\lambda u_3 q;
 \end{eqnarray*}
 \begin{eqnarray*}
 &&R_{1_{\pi}}:\mathcal{A} 1_{\pi}  \longrightarrow \mathcal{A} 1_{\pi},~u_1 1_{\pi}  \longmapsto -\lambda u_1 1_{\pi},~u_2 1_{\pi}  \longmapsto 0 1_{\pi},~u_3 1_{\pi}  \longmapsto -\lambda u_3 1_{\pi}\\
 &&R_{q}:\mathcal{A} q \longrightarrow \mathcal{A} q,~u_1 q  \longmapsto -\lambda u_1 q,~u_2 q  \longmapsto -\lambda u_2 q,~u_3 q  \longmapsto -\lambda u_3 q;
 \end{eqnarray*}
 \begin{eqnarray*}
 &&R_{1_{\pi}}:\mathcal{A} 1_{\pi}  \longrightarrow \mathcal{A} 1_{\pi},~u_1 1_{\pi}  \longmapsto -\lambda u_1 1_{\pi},~u_2 1_{\pi}  \longmapsto -\lambda u_2 1_{\pi},~u_3 1_{\pi}  \longmapsto 0 1_{\pi}\\
 &&R_{q}:\mathcal{A} q \longrightarrow \mathcal{A} q,~u_1 q  \longmapsto 0 q,~u_2 q  \longmapsto -\lambda u_2 q,~u_3 q  \longmapsto 0 q;
 \end{eqnarray*}
 \begin{eqnarray*}
 &&R_{1_{\pi}}:\mathcal{A} 1_{\pi}  \longrightarrow \mathcal{A} 1_{\pi},~u_1 1_{\pi}  \longmapsto 0 1_{\pi},~u_2 1_{\pi}  \longmapsto -\lambda u_2 1_{\pi},~u_3 1_{\pi}  \longmapsto -\lambda u_3 1_{\pi}\\
 &&R_{q}:\mathcal{A} q \longrightarrow \mathcal{A} q,~u_1 q  \longmapsto 0 q,~u_2 q  \longmapsto 0 q,~u_3 q  \longmapsto -\lambda u_3 q;
 \end{eqnarray*}
 \begin{eqnarray*}
 &&R_{1_{\pi}}:\mathcal{A} 1_{\pi}  \longrightarrow \mathcal{A} 1_{\pi},~u_1 1_{\pi}  \longmapsto -\lambda u_1 1_{\pi},~u_2 1_{\pi}  \longmapsto 0 1_{\pi},~u_3 1_{\pi}  \longmapsto -\lambda u_3 1_{\pi}\\
 &&R_{q}:\mathcal{A} q \longrightarrow \mathcal{A} q,~u_1 q  \longmapsto 0 q,~u_2 q  \longmapsto 0 q,~u_3 q  \longmapsto -\lambda u_3 q;
 \end{eqnarray*}
 \begin{eqnarray*}
 &&R_{1_{\pi}}:\mathcal{A} 1_{\pi}  \longrightarrow \mathcal{A} 1_{\pi},~u_1 1_{\pi}  \longmapsto -\lambda u_1  1_{\pi},~u_2 1_{\pi}  \longmapsto -\lambda u_2 1_{\pi},~u_3 1_{\pi}  \longmapsto -\lambda u_3 1_{\pi}\\
 &&R_{q}:\mathcal{A} q \longrightarrow \mathcal{A} q,~u_1 q  \longmapsto 0 q,~u_2 q  \longmapsto 0 q,~u_3 q  \longmapsto -\lambda u_3 q;
 \end{eqnarray*}
 \begin{eqnarray*}
 &&R_{1_{\pi}}:\mathcal{A} 1_{\pi}  \longrightarrow \mathcal{A} 1_{\pi},~u_1 1_{\pi}  \longmapsto -\lambda u_1 1_{\pi},~u_2 1_{\pi}  \longmapsto -\lambda u_2 1_{\pi},~u_3 1_{\pi}  \longmapsto -\lambda u_3 1_{\pi}\\
 &&R_{q}:\mathcal{A} q \longrightarrow \mathcal{A} q,~u_1 q  \longmapsto 0 q,~u_2 q  \longmapsto -\lambda u_2 q,~u_3 q  \longmapsto -\lambda u_3 q.
 \end{eqnarray*}

 (3) Let $\mathcal{A}$ be the 4-dimensional algebra defined in Example \ref{ex:15.17}. For all $\varphi\in \pi$, define $\D_{\varphi}: \mathcal{A}\varphi\lr \mathcal{A}\varphi$ by $\D(h\varphi)=h\varphi\o h\varphi$, $\forall~h\in \mathcal{A}$. Then $(\mathcal{A}[\pi], \D, R)$ is a Rota-Baxter semi-Hopf T-algebra of weight $\lambda$ with $\{R_{\varphi}\}$ given by
 \begin{eqnarray*}
 &&R_{1_{\pi}}:\mathcal{A} 1_{\pi}  \longrightarrow \mathcal{A} 1_{\pi},~u_1 1_{\pi}  \longmapsto 0 1_{\pi},~u_2 1_{\pi}  \longmapsto 0 1_{\pi},~u_3 1_{\pi}  \longmapsto -\lambda u_3 1_{\pi},~u_4 1_{\pi}  \longmapsto -\lambda u_4 1_{\pi}\\  &&R_{q}:\mathcal{A} q \longrightarrow \mathcal{A} q,~u_1 q  \longmapsto -\lambda u_1 q,~u_2 q  \longmapsto -\lambda u_2 q,~u_3 q  \longmapsto -\lambda u_3 q,~u_4 q  \longmapsto -\lambda u_4 q;
 \end{eqnarray*}
 \begin{eqnarray*}
 &&R_{1_{\pi}}:\mathcal{A} 1_{\pi}  \longrightarrow \mathcal{A} 1_{\pi},~u_1 1_{\pi}  \longmapsto -\lambda u_1  1_{\pi},~u_2 1_{\pi}  \longmapsto -\lambda u_2  1_{\pi},~u_3 1_{\pi}  \longmapsto 0 1_{\pi},~u_4 1_{\pi}  \longmapsto 0 1_{\pi}\\  &&R_{q}:\mathcal{A} q \longrightarrow \mathcal{A} q,~u_1 q  \longmapsto -\lambda u_1 q,~u_2 q  \longmapsto -\lambda u_2 q,~u_3 q  \longmapsto -\lambda u_3 q,~u_4 q  \longmapsto -\lambda u_4 q;
 \end{eqnarray*}
 \begin{eqnarray*}
 &&R_{1_{\pi}}:\mathcal{A} 1_{\pi}  \longrightarrow \mathcal{A} 1_{\pi},~u_1 1_{\pi}  \longmapsto -\lambda u_1  1_{\pi},~u_2 1_{\pi}  \longmapsto -\lambda u_2  1_{\pi},~u_3 1_{\pi}  \longmapsto -\lambda u_3 1_{\pi},~u_4 1_{\pi}  \longmapsto -\lambda u_4 1_{\pi}\\
 &&R_{q}:\mathcal{A} q \longrightarrow \mathcal{A} q ,~u_1 q  \longmapsto 0 q,~u_2 q  \longmapsto 0 q,~u_3 q  \longmapsto -\lambda u_3 q,~u_4 q  \longmapsto -\lambda u_4 q.
 \end{eqnarray*}
 \end{ex}

 \subsection{Rota-Baxter Hopf T-coalgebras}

 \begin{defi}\label{de:5.1} Let $\pi$ be a semigroup and $\gamma\in K$ be given. A {\bf T-coalgebra} \cite{Tu,Tu1} is a pair $(\{C_{\varphi}\}_{\varphi\in\pi}$, $\{\Delta_{p,q}\}_{p,q\in\pi})$, where $\{C_{\varphi}\}_{\varphi\in\pi}$ is a family of vector spaces together with a family of linear maps $\{\Delta_{p,q}: C_{pq} \lr C_{p}\otimes C_{q}\}_{p,q\in\pi}$ such that
 \begin{eqnarray}
 &(\Delta_{p,q}\otimes \id_{C_t})\Delta_{pq,t}=(\id_{C_p}\otimes \Delta_{q,t})\Delta_{p,qt},&\label{eq:5.1}
 \end{eqnarray}
 where $p, q, t\in \pi$. A {\bf Rota-Baxter T-coalgebra of weight $\gamma$ } is a quadruple $(\{C_{\varphi}\}_{\varphi\in\pi}$, $\{\Delta_{p,q}\}_{p,q\in\pi}$, $\{Q_{\varphi}\}_{\varphi\in\pi}, \gamma)\  (abbr.(C, Q))$, where $(\{C_{\varphi}\}_{\varphi\in\pi}, \{\Delta_{p,q}\}_{p,q\in\pi})$ is a T-coalgebra and $\{Q_{\varphi}: C_{\varphi}\lr C_{\varphi}\}_{\varphi\in\pi}$ is a family of linear maps such that
 \begin{eqnarray*}
 &(Q_{p}\otimes Q_{q})\Delta_{p,q}=(\id_{C_p}\otimes Q_{q})\Delta_{p,q}Q_{p,q}+(Q_{p}\otimes \id_{C_q})\Delta_{p,q}Q_{p,q}+\gamma\Delta_{p,q}Q_{p,q},&\label{eq:5.3}
 \end{eqnarray*}
 where $p, q\in \pi$.

 If moreover $\pi$ is a monoid with unit $1$ and there is a linear map $\varepsilon: C_{1} \lr K$ such that
 \begin{eqnarray}
 &(\id_{C_p}\otimes\varepsilon)\Delta_{p,1}=\id_{C_p}=(\varepsilon\otimes \id_{C_p})\Delta_{1,p},&\label{eq:5.2}
 \end{eqnarray}
 then we call $(C, Q)$ {\bf counital}.
 \end{defi}

 \begin{rmk} (1) If the semigroup $\pi$ contains a single element $e$, then Rota-Baxter T-coalgebra $(C_{1}, \D_{1,1}, Q_{1}, \g)$ is exactly the Rota-Baxter coalgebra of weight $\l$ introduced in \cite{E06} and studied in \cite{ML}.

 (2) If $(C,\D,\v)$ is a coassociative coalgebra, then Rota-Baxter T-coalgebra $(\{C_{\varphi}=C\}, \{\D_{p, q}=\D\}, \{Q_{\varphi}\}, \g)$ is called a {\bf  Rota-Baxter family coalgebra of weight $\g$}, which is a dual to the Rota-Baxter family algebra of weight $\l$ introduced in \cite[Proposition 9.1]{EFGBP}. 
 \end{rmk}

 The following definitions and two propositions are the corresponding parts above.

 \begin{defi}\label{de:5.6} Let $\pi$ be a semigroup. A {\bf dendriform T-coalgebra} is a family of vector spaces $\{C_{\varphi}\}_{\varphi\in\pi}$ with a family of binary operations $\{\Delta_{\prec p,q},\Delta_{\succ p,q}:C_{pq}\lr C_{p}\otimes C_{q}\}_{p,q\in\pi}$
 satisfying the following conditions (for all $p,q,t\in \pi$)
 \begin{eqnarray}
 &(\Delta_{\prec p,q}\otimes \id_{C_t})\Delta_{\prec pq,t}=(\id_{C_p}\otimes\Delta_{\prec q,t})\Delta_{\prec p,qt}+(\id_{C_p}\otimes\Delta_{\succ q,t})\Delta_{\prec p,qt},&\label{eq:5.18}\\
 &(\Delta_{\succ p,q}\otimes \id_{C_t})\Delta_{\prec pq,t}=(\id_{C_p}\otimes\Delta_{\prec q,t})\Delta_{\succ p,qt},&\label{eq:5.19}\\
 &(\Delta_{\prec p,q}\otimes \id_{C_t})\Delta_{\succ pq,t}+(\Delta_{\succ p,q}\otimes \id_{C_t})\Delta_{\succ pq,t}=(\id_{C_p}\otimes \Delta_{\succ q,t})\Delta_{\succ p,qt}.&\label{eq:5.20}
 \end{eqnarray}
 For simplicity, we denote it by $(\{C_{\varphi}\}_{\varphi\in\pi}, \{\Delta_{\prec p,q}\}_{p,q\in\pi}, \{\Delta_{\succ p,q}\}_{p,q\in\pi})$ (abbr. $(C,\Delta_{\prec}, \Delta_{\succ})$).
 \end{defi}

 \begin{rmk} If $\pi$ contains a single element $1$, then $(C_{1}, \Delta_{\prec p,q}=\Delta_{\prec 1,1}, \Delta_{\succ p,q}=\Delta_{\succ 1,1})$ is dendriform coalgebra in \cite{Fo}.
 \end{rmk}

 \begin{defi}\label{de:5.7} Let $\pi$ be a semigroup. A {\bf tridendriform T-coalgebra} is a family of vector spaces $\{C_{\varphi}\}_{\varphi\in\pi}$ with a family of binary operations $\{\Delta_{\prec p,q},\Delta_{\succ p,q},\Delta_{\bullet p,q}:C_{pq}\lr C_{p}\otimes C_{q}\}_{ p,q\in\pi}$
 satisfying the following conditions (for all $p,q,t\in \pi$)
 \begin{eqnarray}
 &(\Delta_{\prec p,q}\otimes \id_{C_t})\Delta_{\prec pq,t}=(\id_{C_p}\otimes\Delta_{\prec q,t})\Delta_{\prec p,qt}+(\id_{C_p}\otimes\Delta_{\succ q,t})\Delta_{\prec p,qt}+(\id_{C_p}\otimes\Delta_{\bullet q,t})\Delta_{\prec p,qt},\qquad &\label{eq:5.21}\\
 &(\Delta_{\succ p,q}\otimes \id_{C_t})\Delta_{\prec pq,t}=(\id_{C_p}\otimes\Delta_{\prec q,t})\Delta_{\succ p,qt},&\label{eq:5.22}\\
 &(\Delta_{\prec p,q}\otimes \id_{C_t})\Delta_{\succ pq,t}+(\Delta_{\succ p,q}\otimes \id_{C_t})\Delta_{\succ pq,t}+(\Delta_{\bullet p,q}\otimes \id_{C_t})\Delta_{\succ pq,t}=(\id_{C_p}\otimes \Delta_{\succ q,t})\Delta_{\succ p,qt},\qquad &\label{eq:5.23}\\
 &(\Delta_{\succ p,q}\otimes \id_{C_t})\Delta_{\bullet pq,t}=(\id_{C_p}\otimes \Delta_{\bullet q,t})\Delta_{\succ p,qt},&\label{eq:5.24}\\
 &(\Delta_{\prec p,q}\otimes \id_{C_t})\Delta_{\bullet pq,t}=(\id_{C_p}\otimes \Delta_{\succ q,t})\Delta_{\bullet p,qt},&\label{eq:5.25}\\
 &(\Delta_{\bullet p,q}\otimes \id_{C_t})\Delta_{\prec pq,t}=(\id_{C_p}\otimes \Delta_{\prec q,t})\Delta_{\bullet p,qt},&\label{eq:5.26}\\
 &(\Delta_{\bullet p,q}\otimes \id_{C_t})\Delta_{\bullet pq,t}=(\id_{C_p}\otimes \Delta_{\bullet q,t})\Delta_{\bullet p,qt}.&\label{eq:5.27}
 \end{eqnarray}
 For simplicity, we denote it by $(\{C_{\varphi}\}_{\varphi\in\pi}, \{\Delta_{\prec p,q}\}_{p,q\in\pi}, \{\Delta_{\succ p,q}\}_{p,q\in\pi},\{\Delta_{\bullet p,q}\}_{p,q\in\pi})$ (abbr. $(C,\Delta_{\prec}, \Delta_{\succ}, \Delta_{\bullet})$).
 \end{defi}

 \begin{rmk} If $\pi$ contains a single element $1$, then $(C_{1}, \Delta_{\prec p,q}=\Delta_{\prec 1,1}, \Delta_{\succ p,q}=\Delta_{\succ 1,1}, \Delta_{\bullet p,q}=\Delta_{\bullet 1,1})$ is tridendriform coalgebra in \cite{MLY}.
 \end{rmk}

 \begin{pro}\label{pro:5.8} Let $\pi$ be a semigroup. (1) A Rota-Baxter T-coalgebra $(C, Q)$ induces a dendriform T-coalgebra $(\{C_{\varphi}\}_{\varphi\in\pi}, \{\Delta_{\prec p,q}\}_{p,q\in\pi}, \{\Delta_{\succ p,q}\}_{p,q\in\pi})$, where
 \begin{eqnarray*}
 &\Delta_{\prec p,q}(c):=c_{(1,p)}\otimes Q_{q}(c_{(2,q)})+\gamma c_{(1,p)}\otimes c_{(2,q)},\  \Delta_{\succ p,q}(c):=Q_{p}(c_{(1,p)})\otimes c_{(2,q)},&
 \end{eqnarray*}
 here we write $\Delta_{p,q}(c)=c_{(1,p)}\otimes c_{(2,q)}$, for all $c \in C_{pq}$ and $p,q\in \pi$.

 (2) A Rota-Baxter T-coalgebra $(C, Q)$ induces a tridendriform T-coalgebra $(\{C_{\varphi}\}_{\varphi\in\pi}, \{\Delta_{\prec p,q}\}_{p,q\in\pi},$ $\{\Delta_{\succ p,q}\}_{p,q\in\pi},\{\Delta_{\bullet p,q}\}_{p,q\in\pi})$, where
 \begin{eqnarray*}
 &\Delta_{\prec p,q}(c):=c_{(1,p)}\otimes Q_{q}(c_{(2,q)}),\ \Delta_{\succ p,q}(c):=Q_{p}(c_{(1,p)})\otimes c_{(2,q)},\  \Delta_{\bullet p,q}(c):=\gamma c_{(1,p)}\otimes c_{(2,q)},&
 \end{eqnarray*}
 here we write $\Delta_{p,q}(c)=c_{(1,p)}\otimes c_{(2,q)}$, for all $c \in C_{pq}$ and $p,q\in \pi$.
 \end{pro}

 \begin{pro}\label{pro:5.9} Let $\pi$ be a semigroup. (1) Let $(\{C_{\varphi}\}_{\varphi\in\pi}, \{\Delta_{\prec p,q}\}_{p,q\in\pi}, \{\Delta_{\succ p,q}\}_{p,q\in\pi})$ be a dendriform T-coalgebra. Then $(\{C_{\varphi}\}_{\varphi\in\pi},\{\Delta_{\diamond p,q}\}_{ p,q\in\pi})$ is a T-coalgebra, where $\{\Delta_{\diamond p,q}:C_{pq}\lr C_{p}\otimes C_{q}\}_{p,q\in\pi}$,
 \begin{eqnarray*}
 &\Delta_{\diamond p,q}(c):=\Delta_{\prec p,q}(c)+\Delta_{\succ p,q}(c),&
 \end{eqnarray*}
 for all $c \in C_{pq}$ and $p,q\in \pi$.

 (2) Let $(\{C_{\varphi}\}_{\varphi\in\pi}, \{\Delta_{\prec p,q}\}_{p,q\in\pi},\{\Delta_{\succ p,q}\}_{p,q\in\pi}, \{\Delta_{\bullet p,q}\}_{p,q\in\pi})$ be a tridendriform T-coalgebra. Then $(\{C_{\varphi}\}_{\varphi\in\pi},$ $\{\Delta_{\diamond p,q}\}_{ p,q\in\pi})$ is a T-coalgebra, where $\{\Delta_{\diamond p,q}:C_{pq}\lr C_{p}\otimes C_{q}\}_{p,q\in\pi}$,
 \begin{eqnarray*}
 &\Delta_{\diamond p,q}(c):=\Delta_{\prec p,q}(c)+\Delta_{\bullet p,q}(c)+\Delta_{\succ p,q}(c),&
 \end{eqnarray*}
 for all $c \in C_{pq}$ and $p,q\in \pi$.
 \end{pro}

 \begin{defi}\label{de:5.10} Let $\pi$ be a semigroup. A Rota-Baxter T-coalgebra is a {\bf Rota-Baxter semi-Hopf T-coalgebra} denoted by $(C, \mu, Q)$ if every $\{C_{\varphi}\}_{\varphi\in\pi}$ is an algebra with multiplication $\mu_{\varphi}$ such that $\{\Delta_{p,q}\}_{p,q\in\pi}$ are algebra maps and each $Q_{\varphi}$ is a Rota-Baxter operator, i.e.,
 \begin{eqnarray}
 &Q_{\varphi}(a)Q_{\varphi}(b)=Q_{\varphi}\big(Q_{\varphi}(a)b+aQ_{\varphi}(b)+\gamma ab\big),\ \forall a,b\in C_{\varphi}.&\label{eq:5.32}
 \end{eqnarray}
 If, moreover, $\pi$ is a monoid with unit 1, and Rota-Baxter T-coalgebra is counital such that $\varepsilon:C_1\lr K$ is an algebra map, then we call Rota-Baxter semi-Hopf T-coalgebra {\bf counital}.

 Let $\pi$ be a group. A {\bf Rota-Baxter Hopf T-coalgebra} denoted by $(C, \mu, Q, S)$  is a counital Rota-Baxter semi-Hopf T-coalgebra together with a family of linear maps $\{S_{\varphi}:C_{\varphi}\lr C_{\varphi^{-1}}\}_{\varphi\in\pi}$ such that
 \begin{eqnarray}
 &\mu_{\varphi}\circ(S_{\varphi}\otimes \id_{C_\varphi})\circ\Delta_{\varphi^{-1},\varphi}=\eta_{\varphi}\varepsilon=\mu_{\varphi}\circ(\id_{C_\varphi}\otimes S_{\varphi})\circ\Delta_{\varphi,\varphi^{-1}},&\label{eq:5.33}\\
 &Q_{\varphi^{-1}}\circ S_{\varphi}=S_{\varphi}\circ Q_{\varphi}.&\label{eq:5.34}
 \end{eqnarray}
 \end{defi}

 \begin{rmk} (1) Rota-Baxter Hopf T-coalgebra $(C, \mu, Q, S)$ includes a Hopf T-coalgebra $(\{C_{\varphi}\},$ $\{\D_{p,q}\}, \v, \mu_{\varphi}, \{\eta_{\varphi}\}, \varepsilon_{\varphi}, S_{\varphi})$ in \cite{Tu,Tu1}.

 (2) If $(A,\mu,\eta,\Delta,\varepsilon,S)$ is a Hopf algebra, then we call Rota-Baxter Hopf T-coalgebra $(\{A_{\varphi}=A\}, \{\D_{p,q}=\D\}, \{Q_{\varphi}\}, \v, \g, \{\mu_{\varphi}=\mu\}, \{\eta_{\varphi}=\eta\}, \{S_{\varphi}=S\})$ a {\bf co-Rota-Baxter family Hopf algebra of weight $\g$}.

 (3) If $\pi=\{1\}$, then the Rota-Baxter Hopf T-coalgebra $(\{C_{1}\}, \{\D_{1}\}, \{Q_{1}\}, \v, \g, \{\mu_{1}\}, \{\eta_{1}\}, \{S_{1}\})$ is a $(Q_1, Q_1)$-Rota-Baxter Hopf algebra of weight $\g$.
 \end{rmk}

 \section{Other algebraic structures related to Rota-Baxter T-algebras}\label{se:pi} For the possible applications of subsequent research, in this section, we present T-versions of other classical algebraic structures related to Rota-Baxter T-algebras  and also consider the relations among them. If the semigroup $\pi$ is trivial, then T-versions of these algebraic structures cover the classical cases. For example, the pre-Lie T-algebra $(A_{1}, \ast_{1,1})$ is exactly the pre-Lie algebra in the usual case. From now on we assume that the $\pi$ is a commutative semigroup.

 \begin{defi}\label{de:6.5}  A {\bf commutative} T-algebra is a T-algebra $(\{A_{\varphi}\}_{\varphi\in\pi},$ $\{\mu_{p,q}\}_{p,q\in\pi})$ satisfying
 \begin{eqnarray}
 &a\cdot_{p,q} b=b\cdot_{q,p} a,&\label{eq:6.6}
 \end{eqnarray}
 for $a\in A_{p}$, $b\in A_{q}$, and $p,q \in \pi$.
 \end{defi}

 \subsection{Pre-Lie T-algebras from dendriform T-algebras }
 Pre-Lie algebras are also called Vinberg algebras, as they appear under the name ``left-symmetric algebras" in E. B. Vinberg \cite{Vin}'s work on the classification of homogeneous cones. Independently M. Gerstenhaber \cite{Ger} introduced the same structure under the name ``pre-Lie algebras" in the work on Hochschild cohomology and deformations of algebras. For a survey about pre-Lie algebra, see \cite{Man}.

 \begin{defi}\label{de:6.1} A {\bf pre-Lie T-algebra} is a family of vector spaces $\{A_{\varphi}\}_{\varphi\in \pi}$ with a family of binary operations $\{\ast_{p,q}:A_{p}\otimes A_{q}\lr A_{pq}\}_{p,q\in\pi}$ such that
 \begin{eqnarray}
 &a\ast_{p,qt}(b\ast_{q,t} c)-(a\ast_{p,q} b)\ast_{pq,t} c = b\ast_{q,pt}(a\ast_{p,t} c)-(b\ast_{q,p} a)\ast_{qp,t} c,&\label{eq:6.1}
 \end{eqnarray}
 for $a\in A_{p}$, $b\in A_{q}$, $c\in A_{t}$ and $p,q,t\in \pi$. We denote it by $(\{A_{\varphi}\}_{\varphi\in\pi},\{\ast_{p,q}\}_{p,q\in\pi})$.
 \end{defi}

 \begin{pro}\label{pro:6.2} Let $(\{A_{\varphi}\}_{\varphi\in\pi}, \{\prec_{p,q}\}_{p,q\in\pi}, \{\succ_{p,q}\}_{p,q\in\pi})$ be a dendriform T-algebra. Then $(\{A_{\varphi}\}_{\varphi\in\pi},$ $\{\ast_{p,q}\}_{p,q\in\pi})$ is a pre-Lie T-algebra, where $\{\ast_{p,q}:A_{p}\otimes A_{q}\lr A_{pq}\}_{p,q\in\pi}$ defined by
 \begin{eqnarray}
 &a \ast_{p,q} b := a \succ_{p,q}b-b \prec_{q,p}a,&\label{eq:6.2}
 \end{eqnarray}
 for $a\in A_{p}$, $b\in A_{q}$ and $p,q\in \pi$.
 \end{pro}

 \begin{proof}
 For $a\in A_{p}$, $b\in A_{q}$, $c\in A_{t}$ and $p,q,t\in \pi$, we have
 \begin{eqnarray*}
 &&a\ast_{p,qt}(b\ast_{q,t} c)-(a\ast_{p,q} b)\ast_{pq,t} c\\
 &&\qquad\qquad\stackrel{(\ref{eq:6.2})}=
 a\succ_{p,qt}(b\succ_{q,t} c-c\prec_{t,q} b)-(b\succ_{q,t} c-c\prec_{t,q} b)\prec_{qt,p}a\\
 &&\qquad\qquad\qquad-(a\succ_{p,q} b-b\prec_{q,p} a)\succ_{pq,t}c+c\prec_{t,pq}(a\succ_{p,q} b-b\prec_{q,p} a)\\
 &&\qquad\quad\stackrel{(\ref{eq:5.4})-(\ref{eq:5.6})}=
 (a\prec_{p,q}b)\succ_{pq,t} c+(a\succ_{p,q}b)\succ_{pq,t} c-a\succ_{p,tq}(c\prec_{t,q} b)-b\succ_{q,tp} (c\prec_{t,p} a)\\
 &&\qquad\qquad\qquad+c\prec_{t,qp} (b\prec_{q,p}a)+c\prec_{t,qp} (b\succ_{q,p}a)-(a\succ_{p,q} b)\succ_{pq,t}c+(b\prec_{q,p} a)\succ_{qp,t}c\\
 &&\qquad\qquad\qquad+c\prec_{t,pq}(a\succ_{p,q} b)-c\prec_{t,qp}(b\prec_{q,p} a)\\
 &&\qquad\qquad\quad=(a\prec_{p,q}b)\succ_{pq,t} c-a\succ_{p,tq}(c\prec_{t,q} b)-b\succ_{q,tp} (c\prec_{t,p} a)+c\prec_{t,qp} (b\succ_{q,p}a)\\
 &&\qquad\qquad\qquad+(b\prec_{q,p} a)\succ_{qp,t}c+c\prec_{t,pq}(a\succ_{p,q} b).
 \end{eqnarray*}
 Similarly, we have
 \begin{eqnarray*}
 &&b\ast_{q,pt}(a\ast_{p,t} c)-(b\ast_{q,p} a)\ast_{qp,t} c=(b\prec_{q,p}a)\succ_{qp,t} c-b\succ_{q,tp}(c\prec_{t,p} a)\\
 &&-a\succ_{p,tq} (c\prec_{t,q} b)+c\prec_{t,pq} (a\succ_{p,q}b)+(a\prec_{p,q} b)\succ_{pq,t}c+c\prec_{t,qp}(b\succ_{q,p} a).
 \end{eqnarray*}
 Hence $(\{A_{\varphi}\}_{\varphi\in\pi},\{\ast_{p,q}\}_{p,q\in\pi})$ is a pre-Lie T-algebra.
 \end{proof}

 \begin{rmk}\label{rmk:15.20} New constructions of pre-Lie T-algebra can be obtained by Proposition \ref{pro:6.2} and Example \ref{ex:15.12}.
 \end{rmk}

 \begin{ex}\label{ex:15.20a} (1) In general, according to Proposition \ref{pro:6.2}, commutative dendriform T-algebras induce commutative pre-Lie T-algebras. The following example comes from (3) in Example \ref{ex:15.12} which shows that noncommutative dendriform T-algebra can also induce commutative (i.e., $a\ast_{p,q} b=b\ast_{q,p} a)$ pre-Lie T-algebra.

 The new structure of pre-Lie T-algebra on $\mathcal{A}[\pi]$ (where $\mathcal{A}$ is given in Example \ref{ex:15.15}) can be defined by 
 \begin{eqnarray*}
 && u_{1}1_{\pi} \ast  u_{1}q=u_{1}q \ast  u_{1}1_{\pi}=u_{1}q \ast  u_{1}q= -\lambda u_{1} q,\quad u_{1}1_{\pi} \ast  u_{2}q=u_{2}q \ast  u_{1}1_{\pi}=u_{2}1_{\pi} \ast  u_{1}q\\
 &&=u_{1}q \ast  u_{2}1_{\pi}=u_{2}1_{\pi} \ast  u_{2}q=u_{2}q \ast  u_{2}1_{\pi}= -\lambda u_{2} q,u_{2}q \ast  u_{2}q=u_{1}q \ast  u_{2}q=u_{2}q \ast  u_{1}q= -\lambda u_{2} q,\\
 &&u_{1}1_{\pi} \ast  u_{2}1_{\pi}=u_{2}1_{\pi} \ast  u_{1}1_{\pi}=u_{2}1_{\pi} \ast  u_{2}1_{\pi}=-\lambda u_{2} 1_{\pi}, \quad u_{1}1_{\pi} \ast  u_{1}1_{\pi}= -\lambda u_{1} 1_{\pi}.
 \end{eqnarray*}

 (2) The following example comes from (4) in Example \ref{ex:15.12} which shows that noncommutative dendriform T-algebra can also induce noncommutative pre-Lie T-algebra.

 The new structure of pre-Lie T-algebra on $\mathcal{A}[\pi]$ (where $\mathcal{A}$ is given in Example \ref{ex:15.17}) can be defined by 
 \begin{eqnarray*}
 &&u_{1}1_{\pi} \ast  u_{1}q=u_{1}q \ast  u_{1}1_{\pi}=u_{2}q \ast  u_{2}1_{\pi}=u_{1}q \ast  u_{1}q=u_{2}q \ast  u_{2}q=-\lambda u_{1} q,\quad u_{1}1_{\pi} \ast  u_{2}q\\
 &&=u_{2}q \ast  u_{1}1_{\pi}=u_{2}1_{\pi} \ast  u_{1}q=u_{1}q \ast  u_{2}1_{\pi}=u_{1}q \ast  u_{2}q=u_{2}q \ast  u_{1}q= -\lambda u_{2} q,\quad u_{1}1_{\pi} \ast  u_{3}q\\
 &&=u_{3}q \ast  u_{1}1_{\pi}=-u_{4}q \ast  u_{2}1_{\pi}=u_{3}1_{\pi} \ast  u_{1}q=u_{1}q \ast  u_{3}1_{\pi}=u_{2}q \ast  u_{4}1_{\pi}=u_{1}q \ast  u_{3}q=u_{3}q \ast  u_{1}q\\
 &&=u_{2}q \ast  u_{4}q=-u_{4}q \ast  u_{2}q= -\lambda u_{3} q,\quad u_{1}1_{\pi} \ast  u_{4}q=u_{4}q \ast  u_{1}1_{\pi}=-u_{3}q \ast  u_{2}1_{\pi}=u_{2}q \ast  u_{3}1_{\pi}\\
 &&=u_{4}1_{\pi} \ast  u_{1}q=u_{1}q \ast  u_{4}1_{\pi}=u_{1}q \ast  u_{4}q=u_{4}q \ast  u_{1}q=u_{2}q \ast  u_{3}q=-u_{3}q \ast  u_{2}q= -\lambda u_{4} q,\quad\\
 &&u_{2}1_{\pi} \ast  u_{2}q=-\lambda u_{1} q-p_{3}u_{4}q-p_{1}u_{4}q-\frac{2p_{1}p_{2}}{\lambda+p_{2}}u_{3}q,~u_{2}1_{\pi} \ast  u_{3}q= 3\lambda u_{4} q,\\
 &&u_{2}1_{\pi} \ast  u_{4}q= 3 \lambda u_{3} q,~u_{3}1_{\pi} \ast  u_{2}q=(\lambda+2p_{2})u_{4}q+2p_{2}u_{3}q,~u_{1}1_{\pi} \ast  u_{1}1_{\pi}=-\lambda u_{1} 1_{\pi},\\
 &&u_{4}1_{\pi} \ast  u_{2}q=-2(\lambda+p_{2})u_{4}q-(\lambda+2p_{2}) u_{3} q,~u_{1}1_{\pi} \ast  u_{2}1_{\pi}=u_{2}1_{\pi} \ast  u_{1}1_{\pi}= -\lambda u_{2} 1_{\pi},\\
 &&u_{1}1_{\pi} \ast  u_{3}1_{\pi}=u_{3}1_{\pi} \ast  u_{1}1_{\pi}= -\lambda u_{3} 1_{\pi},\quad\
  u_{1}1_{\pi} \ast  u_{4}1_{\pi}=u_{4}1_{\pi} \ast  u_{1}1_{\pi}= -\lambda u_{4} 1_{\pi},\\
 &&u_{2}1_{\pi} \ast  u_{2}1_{\pi}=-\lambda u_{1} 1_{\pi}- 2p_{1}u_{4}1_{\pi}- \frac{2p_{1}p_{2}}{\lambda+p_{2}}u_{3}1_{\pi},\quad u_{2}1_{\pi} \ast  u_{3}1_{\pi}=  3\lambda u_{4} 1_{\pi},~u_{3}1_{\pi} \ast  u_{2}1_{\pi}\\
 &&=(\lambda+2p_{2})u_{4}1_{\pi},~u_{2}1_{\pi} \ast  u_{4}1_{\pi}=3\lambda u_{3} 1_{\pi},~u_{4}1_{\pi} \ast u_{2}1_{\pi}=-2(\lambda+p_{2})u_{4}1_{\pi}-(\lambda+2p_{2})u_{3}1_{\pi}.
 \end{eqnarray*}
 The operations for the remaining cases are 0.
 \end{ex}

 \subsection{Lie T-algebras from T-algebras }

 \begin{defi}\label{de:6.3} A {\bf Lie T-algebra} is a family of vector spaces $\{A_{\varphi}\}_{\varphi\in \pi}$ together with a family of binary operations $\{[,]_{p,q}: A_{p} \otimes A_{q} \lr A_{pq}\}_{p,q\in \pi}$, such that
 \begin{eqnarray}
 &[a,b]_{p,q}+[b,a]_{q,p}=0,&\label{eq:6.3}\\
 &[[a,b]_{p,q},c]_{pq,t}+[[b,c]_{q,t},a]_{qt,p}+[[c,a]_{t,p},b]_{tp,q}=0&\label{eq:6.4}
 \end{eqnarray}
 for $a\in A_{p}$, $b\in A_{q}$, $c\in A_{t}$, and $p,q,t\in \pi$. We denote it by $(\{A_{\varphi}\}_{\varphi\in\pi}, \{[,]_{p,q}\}_{p,q\in\pi})$ (abbr. $(A, [,])$).
 \end{defi}

 \begin{pro}\label{pro:6.20} Let $(\{A_{\varphi}\}_{\varphi\in\pi}, \{\mu_{p,q}\}_{p,q\in\pi})$ be a T-algebra. Then $(A, [,])$ is a Lie T-algebra, where $\{[,]_{p,q}:A_{p}\otimes A_{q}\lr A_{pq}\}_{p,q\in\pi}$
 \begin{eqnarray}
 & [a,b]_{p,q}:=a\cdot_{p,q}b - b\cdot_{q,p} a&\label{eq:7.30}
 \end{eqnarray}
 for $a\in A_{p}$, $b\in A_{q}$ and $p, q\in\pi$.
 \end{pro}

 \begin{proof} For all $a\in A_{p}$, $b\in A_{q}$, $c\in A_{t}$ and $p,q,t\in \pi$, we have
 \begin{eqnarray*}
 [a,b]_{p,q}+[b,a]_{q,p}=a\cdot_{p,q} b-b\cdot_{q,p} a+b\cdot_{q,p} a-a\cdot_{p,q} b=0
 \end{eqnarray*}
 and
 \begin{eqnarray*}
 [[a,b]_{p,q},c]_{pq,t}
 &\stackrel{(\ref{eq:7.30})}=&(a \cdot_{p,q} b- b \cdot_{q,p} a)\cdot_{pq,t} c-c\cdot_{t,pq}(a \cdot_{p,q} b- b \cdot_{q,p} a)\\
 &=&(a \cdot_{p,q} b)\cdot_{pq,t} c- (b \cdot_{q,p} a)\cdot_{qp,t} c-c\cdot_{t,pq}(a \cdot_{p,q} b)+ c\cdot_{t,qp}(b \cdot_{q,p} a).
 \end{eqnarray*}
 Similarly, we can obtain
 \begin{eqnarray*}
 [[b,c]_{q,t},a]_{qt,p}=(b \cdot_{q,t} c)\cdot_{qt,p} a- (c \cdot_{t,q} b)\cdot_{tq,p} a-a\cdot_{p,qt}(b \cdot_{q,t} c)+ a\cdot_{p,tq}(c \cdot_{t,q} b)
 \end{eqnarray*}
 and
 \begin{eqnarray*}
 [[c,a]_{t,p},b]_{tp,q}=(c \cdot_{t,p} a)\cdot_{tp,q} b- (a \cdot_{p,t} c)\cdot_{pt,q} b-b\cdot_{q,tp}(c \cdot_{t,p} a)+ b\cdot_{q,pt}(a \cdot_{p,t} c).
 \end{eqnarray*}
 By Eq.(\ref{eq:5.37}), we have $[[a,b]_{p,q},c]_{pq,t}+[[b,c]_{q,t},a]_{qt,p}+[[c,a]_{t,p},b]_{tp,q}=0$.
 Thus, $(A, [,])$ is a Lie T-algebra.
 \end{proof}

 \begin{rmk}\label{rmk:15.21} Based on Proposition \ref{pro:6.20} and Remark \ref{rmk:15.19}, we can get many new constructions of Lie T-algebra.
 \end{rmk}

 \begin{ex}\label{ex:15.21a} According to Example \ref{ex:15.19a} (2) and Proposition \ref{pro:6.20}, by Eq.(\ref{eq:6.3}), the nonzero operations of Lie T-algebra on $\mathcal{A}[\pi]$ (where $\mathcal{A}$ is given in Example \ref{ex:15.17}) can be defined by 
 \begin{eqnarray*}
 &&\ [u_{2}1_{\pi}, u_{2}q]= -p_{3}u_{4}q-p_{1}u_{4}q-\frac{2p_{1}p_{2}}{\lambda+p_{2}}u_{3}q,~ [u_{2}1_{\pi}, u_{3}q]= 2\lambda u_{4} q=-[u_{2}q, u_{3}q],~ [u_{2}1_{\pi}, u_{4}q]\\
 &&= 2\lambda u_{3} q=-[u_{2}q, u_{4}q],~[u_{3}1_{\pi}, u_{2}q]= 2(\lambda+p_{2})u_{4}q+2p_{2}u_{3}q,~ [u_{4}1_{\pi}, u_{2}q]= -2(\lambda+p_{2})u_{4}q\\
 &&-2p_{2}u_{3}q,[u_{2}1_{\pi}, u_{3}1_{\pi}]=2(\lambda-p_{2})u_{4}1_{\pi},~
  [u_{2}1_{\pi}, u_{4}1_{\pi}]= 2(\lambda+p_{2})u_{4}1_{\pi}+2(2\lambda+p_{2})u_{3}1_{\pi}.
 \end{eqnarray*}
 \end{ex}

 \subsection{Lie T-algebras from pre-Lie T-algebras }

 \begin{pro}\label{pro:6.15} Let $(\{A_{\varphi}\}_{\varphi\in\pi}, \{\ast_{p,q}\}_{p,q\in\pi})$ be a pre-Lie T-algebra. Define
 \begin{eqnarray}
 & [a,b]_{p,q}:=a\ast_{p,q}b - b\ast_{q,p} a,&\label{eq:6.20}
 \end{eqnarray}
 for $a\in A_{p}$, $b\in A_{q}$ and $p,q\in\pi$. Then $(\{A_{\varphi}\}_{\varphi\in\pi}, \{[,]_{p,q}\}_{p,q\in\pi})$ is a Lie T-algebra.
 \end{pro}

 \begin{proof} For all $a\in A_{p}$, $b\in A_{q}$, $c\in A_{t}$ and $p,q,t\in \pi$, we calculate as follows.
 \begin{eqnarray*}
 [a,b]_{p,q}+[b,a]_{q,p}=a\ast_{p,q} b- b\ast_{q,p} a+b\ast_{q,p} a-a\ast_{p,q} b=0.
 \end{eqnarray*}
 Further,
 \begin{eqnarray*}
 [[a,b]_{p,q},c]_{pq,t}
 &\stackrel{(\ref{eq:6.20})}=&(a \ast_{p,q} b- b \ast_{q,p} a)\ast_{pq,t} c-c\ast_{t,pq}(a \ast_{p,q} b- b \ast_{q,p} a)\\
 &=&(a \ast_{p,q} b)\ast_{pq,t} c- (b \ast_{q,p} a)\ast_{qp,t} c-c\ast_{t,pq}(a \ast_{p,q} b)+ c\ast_{t,qp}(b \ast_{q,p} a).
 \end{eqnarray*}
 Likewise, we can obtain
 \begin{eqnarray*}
 [[b,c]_{q,t},a]_{qt,p}=(b \ast_{q,t} c)\ast_{qt,p} a- (c \ast_{t,q} b)\ast_{tq,p} a-a\ast_{p,qt}(b \ast_{q,t} c)+ a\ast_{p,tq}(c \ast_{t,q} b)
 \end{eqnarray*}
 and
 \begin{eqnarray*}
 [[c,a]_{t,p},b]_{tp,q}=(c \ast_{t,p} a)\ast_{tp,q} b- (a \ast_{p,t} c)\ast_{pt,q} b-b\ast_{q,tp}(c \ast_{t,p} a)+ b\ast_{q,pt}(a \ast_{p,t} c).
 \end{eqnarray*}
 By Eq.(\ref{eq:6.1}), we have $[[a,b]_{p,q},c]_{pq,t}+[[b,c]_{q,t},a]_{qt,p}+[[c,a]_{t,p},b]_{tp,q}=0$. Thus, $(\{A_{\varphi}\}_{\varphi\in\pi}$, $\{[,]_{p,q}\}_{p,q\in\pi})$ is a Lie T-algebra.
 \end{proof}

 \begin{rmk}\label{rmk:15.22} According to Proposition \ref{pro:6.15} and Remark \ref{rmk:15.20}, we can get a series of constructions of Lie T-algebra. 
 \end{rmk}

 \subsection{Pre-Lie T-algebras from Rota-Baxter Lie T-algebras } Next we consider the Rota-Baxter T-operators on Lie T-algebras.

 \begin{defi}\label{de:6.4} Let $\lambda \in K $. A {\bf Rota-Baxter Lie T-algebra of weight $\lambda$} is a Lie T-algebra $(\{A_{\varphi}\}_{\varphi\in\pi}, \{[,]_{p,q}\}_{p,q\in\pi})$ endowed with a family of linear maps $\{R_{\varphi}: A_{\varphi}\lr A_{\varphi}\}_{\varphi\in\pi}$, subject to the relation
 \begin{eqnarray}
 &[R_{p}(a),R_{q}(b)]_{p,q} = R_{pq}([a, R_{q}(b)]_{p,q}) + R_{pq}([R_{p}(a), b]_{p,q})+ \lambda R_{pq}([a,b]_{p,q}),&\label{eq:6.5}
 \end{eqnarray}
 for $a\in A_{p}$, $b\in A_{q}$ and $p,q \in \pi$. We denote it by $(\{A_{\varphi}\}_{\varphi\in\pi}, \{[,]_{p,q}\}_{p,q\in\pi}, \{R_{\varphi}\}_{\varphi\in\pi},\lambda)$.
 \end{defi}

 \begin{pro}\label{pro:6.11} Let $(\{A_{\varphi}\}_{\varphi\in\pi}, \{[,]_{p,q}\}_{p,q\in\pi}, \{R_{\varphi}\}_{\varphi\in\pi})$ be a Rota-Baxter Lie T-algebra of weight $0$. Then $(\{A_{\varphi}\}_{\varphi\in\pi}, \{\ast_{p,q}\}_{p,q\in\pi})$ is a pre-Lie T-algebra, where
 \begin{eqnarray}
 & a\ast_{p,q}b=[R_{p}(a),b]_{p,q}.&\label{eq:6.15}
 \end{eqnarray}
 \end{pro}

 \begin{proof} For all $a\in A_{p}$, $b\in A_{q}$, $c\in A_{t}$ and $p,q,t\in \pi$, we can check the result as follows.
 \begin{eqnarray*}
 a \ast_{p,qt} (b \ast_{q,t} c)-(a \ast_{p,q} b) \ast_{pq,t} c
 &\stackrel{(\ref{eq:6.15})}=&[R_{p}(a), [R_{q}(b), c]_{q,t}]_{p,qt}-[R_{pq}([R_{p}(a),b]_{p,q}),c]_{pq,t}\\
 &\stackrel{(\ref{eq:6.5})}=&[R_{p}(a), [R_{q}(b),c]_{q,t}]_{p,qt}+[R_{pq}([a, R_{q}(b)]_{p,q}),c]_{pq,t}\\
 &&-[[R_{p}(a), R_{q}(b)]_{p,q}, c]_{pq,t}\\
 &\stackrel{(\ref{eq:6.3})}=&[R_{p}(a), [R_{q}(b),c]_{q,t}]_{p,qt}-[R_{pq}([ R_{q}(b),a]_{q,p}),c]_{qp,t}\\
 &&+[c,[R_{p}(a),R_{q}(b)]_{p,q}]_{t,pq}\\
 &\stackrel{(\ref{eq:6.4})}=&-[R_{q}(b),[c,R_{p}(a)]_{t,p}]_{q,tp}-[R_{qp}([R_{q}(b), a]_{q,p}), c]_{qp,t}\\
 &\stackrel{(\ref{eq:6.3})}=&[R_{q}(b), [R_{p}(a),c]_{p,t}]_{q,pt}-[R_{qp}([R_{q}(b), a]_{q,p}), c]_{qp,t}\\
 &\stackrel{(\ref{eq:6.15})}=&b \ast_{q,pt} (a \ast_{p,t} c)-(b \ast_{q,p} a) \ast_{qp,t} c,
 \end{eqnarray*}
 finishing the proof.
 \end{proof}

 \subsection{Dendriform T-algebra and Zinbiel T-algebra} Categorical properties of Leibniz algebras is studied by J.-L. Loday in \cite{LO95} and considered in this connection a new object-Zinbiel algebra, which is related to the Koszul dual to the category of Leibniz algebras.

 \begin{defi}\label{de:6.6} A {\bf Zinbiel T-algebra} is a family of vector spaces $\{A_{\varphi}\}_{\varphi\in \pi}$ together with a family of binary operations $\{\star_{p,q}: A_{p} \otimes A_{q} \lr A_{pq}\}_{p,q\in \pi}$ such that
 \begin{eqnarray}
 &a\star_{p,qt}(b\star_{q,t} c)=(a\star_{p,q} b)\star_{pq,t} c+(b\star_{q,p} a)\star_{qp,t}c&\label{eq:6.7}
 \end{eqnarray}
 for $a\in A_{p}$, $b\in A_{q}$, $c\in A_{t}$ and $p,q,t\in \pi$. We denote it by $(\{A_{\varphi}\}_{\varphi\in \pi}, \{\star_{p,q}\}_{p,q\in \pi})$.
 \end{defi}

 \begin{pro}\label{pro:6.7} Let $(\{A_{\varphi}\}_{\varphi\in\pi}, \{\prec_{p,q}\}_{p,q\in\pi}, \{\succ_{p,q}\}_{p,q\in\pi})$ be a commutative dendriform T-algebra in the sense of
 \begin{eqnarray}
 &a \succ_{p,q} b = b \prec_{q,p}a,&\label{eq:6.8}
 \end{eqnarray}
 define
 \begin{eqnarray}
 &a \star_{p,q} b := a \succ_{p,q}b=b \prec_{q,p}a,& \label{eq:6.9}
 \end{eqnarray}
 for $a\in A_{p}$, $b\in A_{q}$ and $p,q\in\pi$. Then $(\{A_{\varphi}\}_{\varphi\in \pi}, \{\star_{p,q}\}_{p,q\in \pi})$ is a Zinbiel T-algebra.
 \end{pro}

 \begin{proof} Since $(\{A_{\varphi}\}_{\varphi\in\pi}, \{\prec_{p,q}\}_{p,q\in\pi}, \{\succ_{p,q}\}_{p,q\in\pi})$ is a dendriform T-algebra, then we have
 \begin{eqnarray*}
 a\succ_{p,qt}(b\succ_{q,t} c)=(a\prec_{p,q} b+a\succ_{p,q} b)\succ_{pq,t}c.
 \end{eqnarray*}
 Hence
 \begin{eqnarray*}
 a\star_{p,qt}(b\star_{q,t} c)=(a\star_{p,q} b)\star_{pq,t} c+(b\star_{q,p} a)\star_{qp,t}c,
 \end{eqnarray*}
 finishing the proof.
 \end{proof}

 \begin{cor}\label{cor:6.8} (1) Let $(A, R)$ be a Rota-Baxter T-algebra. Then $(\{A_{\varphi}\}_{\varphi\in \pi},$ $\{\ast_{p,q}\}_{p,q\in \pi})$ is a pre-Lie T-algebra, where
 \begin{eqnarray}
 &a \ast_{p,q} b:=R_{p}(a) \cdot_{p,q} b- b \cdot_{q,p} R_{p}(a)-\lambda b \cdot_{q,p} a,&\label{eq:6.10}
 \end{eqnarray}
 for $a\in A_{p}$, $b\in A_{q}$ and $p,q\in\pi$.

 (2) Let $(A, R)$ be a Rota-Baxter T-algebra such that $R_{p}(a) \cdot_{p,q} b=b \cdot_{q,p} R_{p}(a)+\lambda b \cdot_{q,p} a$. Then $(\{A_{\varphi}\}_{\varphi\in \pi}, \{\star_{p,q}\}_{p,q\in \pi})$ is a Zinbiel T-algebra, where
 \begin{eqnarray}
 &a \star_{p,q} b:=R_{p}(a) \cdot_{p,q} b&\label{eq:6.11}
 \end{eqnarray}
 for $a\in A_{p}$, $b\in A_{q}$ and $p,q\in\pi$.
 \end{cor}

 \begin{proof} (1) It follows from Proposition \ref{pro:6.2} and Part (2) in Proposition \ref{pro:5.4}.

 (2) It follows from Proposition \ref{pro:6.7}  and Part (2) in Proposition \ref{pro:5.4}.
 \end{proof}

 \begin{lem}\label{lem:6.12} Let $(\{A_{\varphi}\}_{\varphi\in\pi}, \{\star_{p,q}\}_{p,q\in\pi})$ be a Zinbiel T-algebra. Then
 \begin{eqnarray}
 &a\star_{p,qt}(b\star_{q,t} c)=b\star_{q,pt}(a\star_{p,t} c)&\label{eq:6.17}
 \end{eqnarray}
 where $a\in A_{p}$, $b\in A_{q}$, $c\in A_{t}$ and $p,q,t\in \pi$.
 \end{lem}

 \begin{proof} By Eq.(\ref{eq:6.7}), for $a\in A_{p}$, $b\in A_{q}$, $c\in A_{t}$, we have
 \begin{eqnarray*}
 a\star_{p,qt}(b\star_{q,t} c)=(a\star_{p,q} b)\star_{pq,t} c+(b\star_{q,p} a)\star_{qp,t}c,
 \end{eqnarray*}
 and
 \begin{eqnarray*}
 b\star_{q,pt}(a\star_{p,t} c)=(b\star_{q,p} a)\star_{qp,t} c+(a\star_{p,q} b)\star_{pq,t}c.
 \end{eqnarray*}
 Hence, Eq.(\ref{eq:6.17}) holds. This completes the proof.
 \end{proof}

 \begin{pro}\label{pro:6.13} Let $(\{A_{\varphi}\}_{\varphi\in\pi}, \{\star_{p,q}\}_{p,q\in\pi})$ be a Zinbiel T-algebra. Define
 \begin{eqnarray}
 &a \prec_{p,q} b = b \star_{q,p}a,\quad  a \succ_{p,q} b = a \star_{p,q}b,&\label{eq:6.18}
 \end{eqnarray}
 for $a\in A_{p}$, $b\in A_{q}$ and $p,q\in\pi$. Then $(\{A_{\varphi}\}_{\varphi\in\pi}, \{\prec_{p,q}\}_{p,q\in\pi}, \{\succ_{p,q}\}_{p,q\in\pi})$ is a dendriform T-algebra.
 \end{pro}

 \begin{proof} By Lemma \ref{lem:6.12}, $(a\succ_{p,q}b)\prec_{pq,t}c=a\succ_{p,qt}(b\prec_{q,t}c)$ holds. And for all $a\in A_{p}$, $b\in A_{q}$, $c\in A_{t}$, we have
 \begin{eqnarray*}
 (a\prec_{p,q}b)\prec_{pq,t}c
 &\stackrel{(\ref{eq:6.7})}=&(c\star_{t,q}b)\star_{tq,p} a+(b\star_{q,t}c)\star_{qt,p} a\\
 &\stackrel{(\ref{eq:6.18})}=&a\prec_{p,qt}(b\prec_{q,t}c+b\succ_{q,t}c),
 \end{eqnarray*}
 and
 \begin{eqnarray*}
 (a\prec_{p,q}b+a\succ_{p,q}b)\succ_{pq,t}c
 &\stackrel{(\ref{eq:6.18})}=&(b\star_{q,p}a)\star_{qp,t}c+(a\star_{p,q}b)\star_{pq,t}c\\
 &\stackrel{(\ref{eq:6.7})}=&a\star_{p,qt}(b\star_{q,t}c)\\
 &\stackrel{(\ref{eq:6.18})}=&a\succ_{p,qt}(b\succ_{q,t}c).
 \end{eqnarray*}
 Thus, $(\{A_{\varphi}\}_{\varphi\in\pi}, \{\prec_{ p,q}\}_{p,q\in\pi}, \{\succ_{ p,q}\}_{p,q\in\pi})$ is a dendriform T-algebra.
 \end{proof}

 \begin{pro}\label{pro:6.14} Let $(\{A_{\varphi}\}_{\varphi\in\pi}, \{\star_{p,q}\}_{p,q\in\pi})$ be a Zinbiel T-algebra. Define
 \begin{eqnarray}
 &a \diamond_{p,q} b =  a \star_{p,q} b + b \star_{q,p} a &\label{eq:6.19}
 \end{eqnarray}
 for $a\in A_{p}$, $b\in A_{q}$ and $p,q\in\pi$. Then $(\{A_{\varphi}\}_{\varphi\in\pi}, \{\diamond_{p,q}\}_{p,q\in\pi})$ is a T-algebra.
 \end{pro}

 \begin{proof} It follows from Proposition  \ref{pro:6.13} and Part (1) in Proposition  \ref{pro:5.5}.
 \end{proof}

 \subsection{Poisson T-algebras from pre-Poisson T-algebras and pre-Poisson T-algebras from Rota-Baxter Poisson T-algebras } Pre-Poisson algebra was proposed by Aguiar in \cite{A00} by combining Zinbiel algebra and pre-Lie algebra. We extend it to the T-version. Further we combine the Rota-Baxter operator with Poisson algebra.

 \begin{defi}\label{de:6.9}  A {\bf Poisson T-algebra} is a triple $(\{A_{\varphi}\}_{\varphi\in\pi},$ $\{\mu_{p,q}\}_{p,q\in\pi},\{[,]_{p,q}\}_{p,q\in\pi})$ where $(\{A_{\varphi}\}_{\varphi\in\pi},$ $\{\mu_{p,q}\}_{p,q\in\pi})$ is a commutative T-algebra, $(\{A_{\varphi}\}_{\varphi\in\pi},$ $\{[,]_{p,q}\}_{p,q\in\pi})$ is a Lie T-algebra and the following condition holds
 \begin{eqnarray}
 &[a,b\cdot_{q,t}c]_{p,qt}=[a,b]_{p,q}\cdot_{pq,t} c+b\cdot_{q,pt}[a,c]_{p,t},&\label{eq:6.12}
 \end{eqnarray}
 for $a\in A_{p}$, $b\in A_{q}$, $c\in A_{t}$ and $p,q,t\in \pi$.
 \end{defi}

 \begin{defi}\label{de:6.10} A {\bf pre-Poisson T-algebra} is a triple $(\{A_{\varphi}\}_{\varphi\in\pi},\{\star_{p,q}\}_{p,q\in\pi},\{\ast_{p,q}\}_{p,q\in\pi})$ where $(\{A_{\varphi}\}_{\varphi\in\pi},$ $\{\star_{p,q}\}_{p,q\in\pi})$ is a  Zinbiel T-algebra, $(\{A_{\varphi}\}_{\varphi\in\pi},\{\ast_{p,q}\}_{p,q\in\pi})$ is a pre-Lie T-algebra and the following conditions hold
 \begin{eqnarray}
 &(a\ast_{p,q} b)\star_{pq,t} c-(b\ast_{q,p} a)\star_{qp,t} c=a\ast_{p,qt}(b\star_{q,t}c) -b\star_{q,pt}(a\ast_{p,t}c),&\label{eq:6.13}\\
 &(a\star_{p,q} b)\ast_{pq,t} c+(b\star_{q,p} a)\ast_{qp,t} c=a\star_{p,qt}(b\ast_{q,t}c) +b\star_{q,pt}(a\ast_{p,t}c),&\label{eq:6.14}
 \end{eqnarray}
 for $a\in A_{p}$, $b\in A_{q}$, $c\in A_{t}$ and $p,q,t\in \pi$.
 \end{defi}

 \begin{pro}\mlabel{pro:xz08.24} Let $(\{A_{\varphi}\}_{\varphi\in\pi}, \{\star_{p,q}\}_{p,q\in\pi}, \{\ast_{p,q}\}_{p,q\in\pi})$ be a pre-Poisson $T$-algebra. Define
 \begin{eqnarray*}
 &a\diamond_{p,q}b:=a\star_{p,q}b+b\star_{q,p}a,\quad [a,b]_{p,q}:=a\ast_{p,q}b -b\ast_{q,p}a,&\mlabel{eq:08.1655}
 \end{eqnarray*}
 for $a\in A_{p}$, $b\in A_{q}$ and $p,q\in\pi$. Then $(\{A_{\varphi}\}_{\varphi\in\pi},$ $\{\diamond_{p,q}\}_{p,q\in\pi},\{[,]_{p,q}\}_{p,q\in\pi})$ is a Poisson $T$-algebra.
 \end{pro}

 \begin{proof} Eq.(\ref{eq:6.12}) can be verified as follows. For all $a\in A_p, b\in A_q$ and $c\in A_t$, we have
 \begin{eqnarray*}
 &&[a,b]_{p,q}\diamond_{pq,t}c+b\diamond_{q,pt}[a,c]_{p,t}\\ 
 &&\qquad=(a\ast_{p,q}b)\star_{pq,t}c-(b\ast_{q,p} a)\star_{qp,t}c +c\star_{t,pq}(a\ast_{p,q}b)-c\star_{t,qp}(b\ast_{q,p}a)\\
 &&\qquad\quad+b\star_{q,pt}(a\ast_{p,t}c)-b\star_{q,tp}(c\ast_{t,p} a)
 +(a\ast_{p,t}c)\star_{pt,q}b-(c\ast_{t,p} a)\star_{tp,q}b\\
 &&\quad\stackrel{(\mref{eq:6.13})}=a\ast_{p,qt}(b\star_{q,t}c)-b\star_{q,pt}(a\ast_{p,t}c)
 +c\star_{t,pq}(a\ast_{p,q}b)-c\star_{t,qp}(b\ast_{q,p}a)\\
 &&\qquad\quad+b\star_{q,pt}(a\ast_{p,t}c)-b\star_{q,tp}(c\ast_{t,p} a)
 +a\ast_{p,tq}(c\star_{t,q}b)-c\star_{t,pq}(a\ast_{p,q}b)\\
 &&\qquad=a\ast_{p,qt}(b\star_{q,t}c)-c\star_{t,qp}(b\ast_{q,p}a)
 -b\star_{q,tp}(c\ast_{t,p} a)+a\ast_{p,tq}(c\star_{t,q}b)\\
 &&\quad\stackrel{(\mref{eq:6.14})}=a\ast_{p,qt}(b\star_{q,t}c)+a\ast_{p,tq}(c\star_{t,q}b)
 -(b\star_{q,t}c)\ast_{qt,p} a-(c\star_{t,q}b)\ast_{tq,p}a\\ 
 &&\qquad=[a,b\diamond_{q,t}c]_{pq,t}.
 \end{eqnarray*}
 Then by Propositions \mref{pro:6.15}, \mref{pro:6.13} and Part (1) in Proposition \mref{pro:5.5}, the proof is finished.
 \end{proof}

 Based on Definitions \ref{de:5.14}, \ref{de:6.3} and \ref{de:6.9}, we introduce the following notion.

 \begin{defi}\label{de:6.10C} A {\bf Rota-Baxter Poisson T-algebra of weight $\lambda$} is a triple $(\{A_{\varphi}\}_{\varphi\in\pi}$, $\{\mu_{p,q}\}_{p,q\in\pi}$, $\{[,]_{p,q}\}_{p,q\in\pi})$ endowed with a family of linear maps $\{R_{\varphi}: A_{\varphi}\lr A_{\varphi}\}_{\varphi\in \pi}$ such that

 (1) $(\{A_{\varphi}\}_{\varphi\in\pi}, \{\mu_{p,q}\}_{p,q\in\pi}, \{[,]_{p,q}\}_{p,q\in\pi})$ is a Poisson T-algebra,

 (2) $(\{A_{\varphi}\}_{\varphi\in\pi},$ $\{\mu_{p,q}\}_{p,q\in\pi},\{R_{\varphi}\}_{\varphi\in \pi})$ is a Rota-Baxter T-algebra of weight $\lambda$,

 (3) $(\{A_{\varphi}\}_{\varphi\in\pi}, \{[,]_{p,q}\}_{p,q\in\pi}, \{R_{\varphi}\}_{\varphi\in \pi})$ is a Rota-Baxter Lie T-algebra of weight $\lambda$.

 We denote it by $(\{A_{\varphi}\}_{\varphi\in\pi}, \{\mu_{p,q}\}_{p,q\in\pi}, \{[,]_{p,q}\}_{p,q\in\pi},\{R_{\varphi}\}_{\varphi\in\pi},\lambda)$.
 \end{defi}

 \begin{rmk} We call $(A_{1}, \mu_{1,1}, [,]_{1,1}, R_{1}, \lambda)$ {\bf Rota-Baxter Poisson algebra of weight $\lambda$}.
 \end{rmk}

 \begin{pro}\label{pro:6.16} Let $(\{A_{\varphi}\}_{\varphi\in\pi}, \{\mu_{p,q}\}_{p,q\in\pi}, \{[,]_{p,q}\}_{p,q\in\pi},\{R_{\varphi}\}_{\varphi\in\pi})$ be a Rota-Baxter Poisson T-algebra of weight $0$. Define
 \begin{eqnarray}
 &a\star_{p,q}b:=R_{p}(a)\cdot_{p,q}b,\quad a\ast_{p,q}b:=[R_{p}(a),b]_{p,q},&\label{eq:6.21}
 \end{eqnarray}
 for $a\in A_{p}$, $b\in A_{q}$ and $p,q\in\pi$. Then $(\{A_{\varphi}\}_{\varphi\in\pi}, \{\star_{p,q}\}_{p,q\in\pi}, \{\ast_{p,q}\}_{p,q\in\pi})$ is a pre-Poisson T-algebra.
 \end{pro}

 \begin{proof} Based on Proposition \ref{pro:6.11} and Part (2) in Corollary \ref{cor:6.8}, we only need to check the following two conditions.
 We first prove Eq.(\ref{eq:6.13}) as follows.
 \begin{eqnarray*}
 (a\ast_{p,q} b)\star_{pq,t} c-(b\ast_{q,p} a)\star_{qp,t} c
 &\stackrel{(\ref{eq:6.21})}=&R_{pq}([R_{p}(a),b]_{p,q}) \cdot_{pq,t} c-R_{qp}([R_{q}(b),a]_{q,p}) \cdot_{qp,t} c\\
 &\stackrel{(\ref{eq:6.3})}=&R_{pq}([R_{p}(a),b]_{p,q}) \cdot_{pq,t} c+R_{pq}([a,R_{q}(b)]_{p,q}) \cdot_{pq,t} c\\
 &\stackrel{(\ref{eq:6.5})}=&[R_{p}(a),R_{q}(b)]_{p,q} \cdot_{pq,t} c+R_{q}(b) \cdot_{q,pt} [R_{p}(a),c]_{p,t}\\
 &&-R_{q}(b) \cdot_{q,pt} [R_{p}(a),c]_{p,t}\\
 &\stackrel{(\ref{eq:6.12})}=&[R_{p}(a),R_{q}(b) \cdot_{q,t} c]_{p,qt}-R_{q}(b) \cdot_{q,pt} [R_{p}(a),c]_{p,t}\\
 &\stackrel{(\ref{eq:6.21})}=&a\ast_{p,qt} (b\star_{q,t}c) -b\star_{q,pt}(a\ast_{p,t}c).
 \end{eqnarray*}
 Eq.(\ref{eq:6.14}) can be checked similarly. These finish the proof. 
 \end{proof}

 \section{Further research}\label{se:future}
 \begin{itemize}
   \item In \cite{GL}, the authors gave a broad study of representation and module theory of Rota-Baxter algebras. They provided the equivalent characterizations of Rota-Baxter modules by using quasi-idempotency and the ring of Rota-Baxter operators. Furthermore in \cite{ZGZ}, many examples of Rota-Baxter paired modules were obtained from the theory of Hopf algebras. In the forthcoming paper \cite{MLCW}, we will present the (co)representation theory of Rota-Baxter T-(co)algebras and introduce the notion of Rota-Baxter T-(co)modules and investigate their properties, especially pay more attention to the relation between Rota-Baxter operators and Turaev's Hopf group (co)algebras.

   \item A Novikov-Poisson algebra introduced in \cite{Xu} is a triple $(A, \cdot, *)$ such that $(A, \cdot)$ is a commutative associative algebra, $(A, *)$ is a Novikov algebra and two compatible conditions hold. In \cite{MLiL}, we will study the T-version of Novikov(-Poisson) algebra and explore the Gel'fand-Dorfman theorem on T-algebras. Furthermore, the properties of infinitesimal T-bialgebra (generalizing infinitesimal bialgebra studied in \cite{Ag04}) will be investigated.
   \item In \cite{YW}, the authors combined the Hom-algebra structure introduced in \cite{MS2} and the structure of Turaev's Hopf group coalgebra in \cite{Tu,Tu1}, and then introduced the notion of Hom-Hopf T-coalgebra and gave the construction of Drinfeld double. The structure of BiHom-algebra was invented in \cite{GMMP} generalizing the Hom-algebra. It is natural to consider the BiHom-version of Hopf T-(co)algebra.
 \end{itemize}

\section*{Acknowledgment} The authors are deeply indebted to Professor Vsevolod Gubarev for bringing our attention to refs\cite{AB,de,G} and to the anonymous referee for his/her useful suggestions to improve the original manuscript. This work is supported by Natural Science Foundation of Henan Province (No. 212300410365) and National Natural Science Foundation of China (Nos. 11771069, 11871144).


\begin{thebibliography}{99} \small
 

 \bibitem{AB} H. H. An, C. M. Bai, From Rota-Baxter algebras to pre-Lie algebras. {\em J. Phys. A} {\bf 41}(2008), 015201, 19pp.

 \bibitem{A00} M. Aguiar, Pre-Poison algebras. {\em Lett. Math. Phys.}, {\bf 54} (2000), 263-277.

 \bibitem{Ag04} M. Aguiar, Infinitesimal bialgebras, pre-Lie and dendriform algebras. Hopf algebras, 1-33, {\em Lecture Notes in Pure and Appl. Math.}, {\bf 237}, Dekker, New York, 2004.

 \bibitem{A03} M. Aguiar, Dendriform algebras relative to a semigroup. {\em SIGMA Symmetry Integrability Geom. Methods Appl.} {\bf 16} (2020), Paper No. 066, 15 pp.

 \bibitem{AS} N. Andruskiewitsch, H. J. Schneider, On the classification of finite-dimensional pointed Hopf algebras. {\em Ann. Math.} {\bf 171}(2010), 375-417.

 \bibitem{At} F. V. Atkinson, Some aspects of Baxter's functional equation. {\em J. Math. Anal. Appl.} {\bf 7} (1963), 1-30.

 \bibitem{Ba60} G. Baxter, An analytic problem whose solution follows from a simple algebraic identity. {\em Pacific J. Math.} {\bf 10} (1960), 731-742. 

 \bibitem{BFVV} M. Buckleya, T. Fieremansb, C. Vasilakopoulouc, J. Vercruysse, A Larson-Sweedler theorem for Hopf $\mathcal{V}$-categories. {\em Adv. Math.} {\bf 376} (2021), 107456.

 \bibitem{CLZZ} D. Chen, Y. F. Luo, Y. Zhang, Y. Y. Zhang, Free $\Omega$-Rota-Baxter algebras and Gr\"obner-Shirshov bases. Internat. {\em J. Algebra Comput.} {\bf 30} (2020), 1359-1373.

 \bibitem{CK} A. Connes, D. Kreimer, Renormalization in quantum field theory and the Riemann-Hilbert problem. I. The Hopf algebra structure of graphs and the main theorem. {\em Comm. Math. Phys.} {\bf 210} (2000), 249-273.

 \bibitem{de} S. L. de Braganca, Finite dimensional Baxter algebras. {\em Stud. Appl. Math. (1)} {\bf 54}(1975), 75-89.

 \bibitem{E06} K. Ebrahimi-Fard, Rota-Baxter algebras and the Hopf algebra of renormalization. {\em Ph, D. Thesis, Bonn University}, 2006.

 \bibitem{EFGBP} K. Ebrahimi-Fard, J. Gracia-Bondia, F. Patras, A Lie theoretic approach to renormalization. {\em Comm. Math. Phys.} {\bf 276} (2007), 519-549.

 \bibitem{Fo} L. Foissy, Bidendriform bialgebras, trees and free quasi-symmetric functions. {\em J. Pure Appl. Algebra} {\bf 209} (2007), 439-459.

 \bibitem{Ger}  M. Gerstenhaber, The cohomology structure of an associative ring, {\em Ann. of Math.} {\bf 78} (1963), 267-288.

 \bibitem{Go} M. Goncharov, Rota-Baxter operators on cocommutative Hopf algebras.  {\em J. Algebra} {\bf 582} (2021), 39-56.

 \bibitem{GMMP} G. Graziani, A. Makhlouf, C. Menini, F. Panaite, BiHom-Associative algebras, BiHom-Lie algebras and BiHom-Bialgebras.  {\em SIGMA} {\bf 11} (2015),086,34 pages.

 \bibitem{G} V. Gubarev, Rota-Baxter operators on a sum of fields. {\em J. Algebra Appl.} {\bf 19} (2020), 2050118.

 \bibitem{Guo} L. Guo, Operated semigroups, Motzkin paths and rooted trees. {\em J. Algebraic Combin.} {\bf 29} (2009), 35-62.

 \bibitem{Guo1} L. Guo, An introduction to Rota-Baxter algebra. {\em Surveys of Modern Mathematics}, {\bf 4}. International Press, Somerville, MA; Higher Education Press, Beijing, 2012. xii+226 pp.

 \bibitem{GL} L. Guo, Z. Z. Lin, Representations and modules of Rota-Baxter algebras, arXiv: 1905.01531v2.

 \bibitem{GTY2} L. Guo, J.-Y. Thibon, H. Y. Yu, The Hopf algebras of signed permutations, of weak quasi-symmetric functions and of Malvenuto-Reutenauer. {\em Adv. Math.} {\bf 374} (2020), 107341, 34 pp.


 \bibitem{LR04} J. L. Loday, M. Ronco, Trialgebras and families of polytopes. In: {\em Homotopy Theory: Relations with Algebraic Geometry, Group Cohomology, and Algebraic $K$-Theory}, {\bf vol. 346}. Providence, RI: Amer Math Soc, 2004, 369-398.

 \bibitem{LO95} J. L. Loday, Cup product for Leibniz cohomology and dual Leibniz algebras. {\em Math. Scand.} {\bf 77} (1995), 189-196.


 \bibitem{MLiL} T. S. Ma, B. Li, J. Li, Novikov-Poisson T-algebras. Prepared.


 \bibitem{MLY} T. S. Ma, J. Li, H. Y. Yang, Rota-Baxter bialgebras arising from (co-)quasi-idempotent elements. {\em Hacettepe J. Math. Stat.} {\bf 50} (2021), 216-223.

 \bibitem{MLCW} T. S. Ma, J. Li, L. Y. Chen, S. H. Wang, Rota-Baxter operators on Turaev's Hopf group (co)algebras II. Representations. Prepared.

 \bibitem{MLX} T. S. Ma, H. Y. Li, S. X. Xu, Construction of a braided monoidal category for Brzezi\'{n}ski crossed coproducts of Hopf $\pi$-algebras. {\em Colloq. Math.} {\bf 149} (2017), 309-323.

 \bibitem{MLZ} T. S. Ma, H. Y. Li, Z. C. Zhang, Two results on Brzezinski $\pi$-crossed product. {\em Comm. Algebra} {\bf 42}(2014), 2082-2098.

 \bibitem{ML} T. S. Ma, L. L. Liu, Rota-Baxter coalgebras and Rota-Baxter bialgebras. {\em Linear Multilinear Algebra} {\bf 64} (2016), 968-979.

 \bibitem{MMS} T. S. Ma, A. Makhlouf, S. Silvestrov, Rota-Baxter cosystems and coquasitriangular mixed bialgebras. {\em J. Algebra Appl.} {\bf 20} (2021), 2150064.

 \bibitem{MS2} A. Makhlouf, S. D. Silvestrov, Hom-algebras and hom-coalgebras. {\em J. Algebra Appl.} {\bf 9}(2010), 553-589.

 \bibitem{Man} D. Manchon, A short survey on pre-Lie algebras. {\em Noncommutative geometry and physics: renormalisation, motives, index theory}, 89-102, ESI Lect. Math. Phys., Eur. Math. Soc., Z\"{u}rich, 2011.

 \bibitem{Oh1} T. Ohtsuki, Colored ribbon Hopf algebras and universal invariants of framed links. {\em J. Knot Theory Ramifications} {\bf 2} (1993), 211-232.

 \bibitem{Ro1} G. C. Rota, Baxter algebras and combinatorial identities. I, II, {\em Bull. Amer. Math. Soc.} {\bf 75} (1969), 325-329; 330-334.

 \bibitem{Tu} V. G. Turaev, Homotopy quantum field theory. Appendix 5 by M. M\"{u}ger and Appendices 6 and 7 by A. Virelizier. {\em EMS Tracts in Mathematics}, {\bf 10}. European Mathematical Society (EMS), Z\"{u}rich, 2010. xiv+276 pp.

 \bibitem{Tu1} V. G. Turaev, Homotopy field theory in dimension 3 and crossed group-categories, preprint, arXiv:math.GT/0005291.

 \bibitem{Vin} E. B. Vinberg, The theory of homogeneous convex cones, {\em Transl. Moscow Math. Soc.} {\bf 12} (1963), 340-403.

 \bibitem{Vi02} A. Virelizier, Hopf group-coalgebras. {\em J. Pure Appl. Algebra} {\bf 171} (2002), 75-122.

 \bibitem{W09} S. H. Wang, Turaev group coalgebras and twisted Drinfeld double. {\em Indiana Univ. Math. J.} {\bf 58}(2009), 1395-1417.

 \bibitem{Xu} X. P. Xu, Novikov-Poisson algebras. {\em J. Algebra} {\bf 190} (1997), 253-279.

 \bibitem{YW} D. D. Yan, S. H. Wang, Drinfel'd construction for Hom-Hopf T-coalgebras. {\em Internat. J. Math.} {\bf 31} (2020), 2050058, 31 pp.

 \bibitem{GTY1} H. Y. Yu, L. Guo, J.-Y. Thibon, Weak composition quasi-symmetric functions, Rota-Baxter algebras and Hopf algebras. {\em Adv. Math.} {\bf 344} (2019), 1-34.

 \bibitem{ZGG1} T. J. Zhang, X. Gao, L. Guo, Reynolds algebras and their free objects from bracketed words and rooted trees. {\em J. Pure Appl. Algebra} {\bf 225} (2021), 106766.

 \bibitem{ZGG} Y. Zhang, X. Gao, L. Guo, Matching Rota-Baxter algebras, matching dendriform algebras and matching pre-Lie algebras. {\em J. Algebra} {\bf 552} (2020), 134-170.

 \bibitem{ZG01} Y. Y. Zhang, X. Gao, Free Rota-Baxter family algebras and (tri)dendriform family algebras. {\em Pacific J. Math.} {\bf 301} (2019), 741-776.

 \bibitem{ZGM} Y. Y. Zhang,  X. Gao, D. Manchon, Free (tri)dendriform family algebras. {\em J. Algebra} {\bf 547} (2020), 456-493.

 \bibitem{ZM03} Y. Y. Zhang, D. Manchon, Free pre-Lie family algebras. arXiv:2003.00917.

 \bibitem{ZGZ}  H. H. Zheng, L. Guo, L. Zhang, Rota-Baxter paired modules and their constructions from Hopf algebras. {\em J. Algebra} {\bf 559} (2020), 601-624.

 \bibitem{Z3} M. Zunino, Yetter-Drinfeld modules for crossed structures. {\em J. Pure Appl. Algebra} {\bf 193} (2004), 313-343.

 \bibitem{Z2} M. Zunino, Double construction for crossed Hopf coalgebras. {\em J. Algebra} {\bf 278} (2004), 43-75.

 \end{thebibliography}
 \end{document}